\documentclass[9pt,a4paper,leqno]{amsart}

\usepackage{amssymb,amsmath,amsfonts,amsthm,a4}
\usepackage{graphicx, amsbsy, mathrsfs, esint, bbm, dsfont, bm}
\usepackage{color}
\usepackage{mathtools}
\usepackage{cite}
\usepackage{enumerate}
\usepackage{url}

\usepackage[titletoc,title]{appendix}

\usepackage{mathrsfs} % for \mathscr

\usepackage{graphicx}   % need for figures
\usepackage{subfigure}  % use for side-by-side figures
\usepackage{multirow}
\usepackage{slashed} % for slashing operators

\usepackage{xcolor}
\newtheorem{theorem}{Theorem}
\newtheorem{lemma}{Lemma}[section]
\newtheorem{definition}{Definition}

\newtheorem{proposition}{Proposition}[section]

\newtheorem*{assumption*}{Assumption}
\theoremstyle{remark}
\newtheorem{remark}[theorem]{Remark}

\newtheorem*{remarks*}{Remarks}
\newtheorem*{remark*}{Remark}

\newcommand{\R}{\mathbb{R}} % reals
\newcommand{\C}{\mathbb{C}} % complex numbers
\newcommand{\N}{\mathbb{N}} % naturals
\newcommand{\Z}{\mathbb{Z}} 
\newcommand{\T}{\mathbb{T}}

\newcommand{\eps}{\varepsilon} % epsilon
\newcommand{\cip}[2]{\langle #1, #2 \rangle} %duality bracket
\newcommand{\wt}[1]{\widetilde{#1}} % widetilde
\newcommand{\wh}[1]{\widehat{#1}}	% widehat
\newcommand{\bflo}[1]{\big\lfloor #1 \big\rfloor} % big floor integral part
\newcommand{\bcei}[1]{\big\lceil #1 \big\rceil} % big ceil integral part

\newcommand{\weakto}{\rightharpoonup}

\newcommand{\Dh}{|\nabla|}

\newcommand{\comment}[1]{}
\numberwithin{equation}{section}

\newcommand{\EE}{\mathcal{E}}

\newcommand{\cI}{\mathcal{I}}

\newcommand{\dd}{\mathrm{d}}

\newcommand{\norm}[1]{\|#1 \|}
\newcommand{\be}{\begin{equation}}
\newcommand{\ee}{\end{equation}}
\newcommand{\bes}{\begin{equation*}}
\newcommand{\ees}{\end{equation*}}

\newcommand{\Ss}{\mathbb{S}}
\newcommand{\Sb}{{\mathbf{S}}}
\newcommand{\Cb}{{\mathbf{C}}}

\newcommand{\pt}{\partial}
\newcommand{\im}{\mathrm{i}}

\catcode`@=11
\def\section{\@startsection{section}{1}%
  \z@{1.5\linespacing\@plus\linespacing}{.5\linespacing}%
  {\normalfont\bfseries\large\centering}}
\catcode`@=12

\begin{document}

\title[Derivation of the HWM Equation from Calogero--Moser Spin Systems]{Derivation of the Half-Wave Maps Equation from Calogero--Moser Spin Systems}

\author[Enno Lenzmann]{Enno Lenzmann}
\address[Enno Lenzmann]{University of Basel, Department of Mathematics and Computer Science, Spiegelgasse 1, CH-4051 Basel, Switzerland} 
\email{enno.lenzmann@unibas.ch}

\author[J\'er\'emy Sok]{J\'er\'emy Sok}
\address[J\'er\'emy Sok]{University of Basel, Department of Mathematics and Computer Science, Spiegelgasse 1, CH-4051 Basel, Switzerland} 
\email{jeremyvithya.sok@unibas.ch}

\date{}

\begin{abstract}
We prove that the energy-critical half-wave maps equation
$$
\pt_t \Sb =\Sb \times \Dh \Sb, \quad (t,x) \in \R \times \T
$$
arises as an effective equation in the continuum limit of completely integrable Calogero--Moser classical spin systems with inverse square $1/r^2$ interactions on the circle. We study both the convergence to global-in-time weak solutions in the energy class as well as short-time strong solutions of higher regularity. The proofs are based on Fourier methods and suitable discrete analogues of fractional Leibniz rules and Kato--Ponce--Vega commutator estimates. 

In a companion paper, we further extend our arguments to study the real line case and more general spin interactions. 
\end{abstract}

\maketitle

%{\bf Keywords:} Half-Wave Maps Equation, Calogero--Moser spin systems, Haldane--Shastry quantum spin chains.

\section{Introduction and Main Results}

This paper is devoted to the rigorous derivation of the \textbf{half-wave maps equation} 
\be \tag{HWM} \label{eq:HWM}
\pt_t \Sb = \Sb \times \Dh \Sb ,
\ee
which has been recently introduced and studied in \cite{LeSc-18, KrSi-18, GeLe-18, BeKlLa-20, ZhSt-15}. Here the function $\Sb=\Sb(t,x)$ takes values in the standard unit two-sphere $\Ss^2 \subset \R^3$ and the symbol $\times$ denotes the cross product in $\R^3$. As usual, the operator $\Dh=\sqrt{-\Delta}$ denotes the square root of minus the Laplacian. Alternatively, we can write $\Dh = H \pt_x$, where $H$ denotes the Hilbert transform; see below for more details on our notation and sign conventions. Throughout this paper, we will always consider the periodic case when $x \in \T = \R/ 2\pi \Z$; see also our companion work \cite{LeSo-20} for the real line case and further extensions.

Equation \eqref{eq:HWM} is of Hamiltonian nature and it displays a rich list of notable analytic features such as energy-criticality, conformal M\"obius symmetry, Lax pair structure, explicit traveling solitary waves related to minimal surfaces, and $N$-soliton solutions given by rational functions with explicitly known pole dynamics; see \cite{LeSc-18, GeLe-18, BeKlLa-20}. For a recent survey on the (HWM) equation, we refer to \cite{Le-18}.

\subsection{Setting of the Problem}

The starting point of our analysis is the following classical \textbf{Calogero--Moser  (CM) spin system} introduced in the physics literature by Gibbons and Hermsen \cite{GiHe-84}  and independently by Wojciechowski \cite{Wo-85}; see also \cite{BaBeTa-03,KrBaBiTa-95}. In generality, the corresponding Hamiltonian of this classical many-body system for $N \in \N$ particles can be written as
\be\label{eq:H_CM_def}
H^{(\mu)}_{\mathrm{CM}} = \frac{1}{2 \mu} \sum_{j=1}^N p_j^2 + \kappa \sum_{j \neq k}^N (\Sb_j \cdot \Sb_k) V(x_j-x_k)
\ee
with the variables $p_j \in \R$ (momenta) and $x_j \in \R$ (positions) as well as the classical spin variables $\Sb_j = (S_j^1, S_j^2, S_j^3) \in \Ss^2$. Here $\mu > 0$ and $\kappa > 0$ denote a mass parameter and a coupling constant that we keep explicit for the moment. Furthermore, the interaction potential $V$ is given by one of the following choices:
$$
V(x) = \begin{dcases*} \frac{1}{x^2} & (rational case) \\ \frac{a^2}{\sin^2 (a x)} & (trigonometric case) \\ \frac{a^2}{\sinh^2 (ax)} & (hyperbolic case) \\ a^2 \wp(a x) & (elliptic case) \end{dcases*} .
$$
Here $a > 0$ is a constant (chosen below) and $\wp(z) = \wp(z; \omega_1, \omega_2)$ denotes the Weierstrass ellitptic function with suitable complex periods $\omega_1, \omega_2 \in \C$.   As a remarkable fact, it was observed in \cite{GiHe-84} (with formal arguments) that the Hamiltonian $H_{\mathrm{CM}}^{(\mu)}$ yields a {\em completely integrable $N$-body system} in the sense of Liouville. Another striking feature of this classical CM spin system is its formal connection to {\em Haldane--Shastry  (HS) quantum spin chains} \cite{Ha-88, Sh-88}, which are exactly solvable many-body quantum systems that have attracted  a lot of attention.

As mentioned above, we will consider the trigonometric case with positions $x_j$ subject to the periodicity condition
$$
x_j \in \T = \R / 2 \pi \Z.
$$
[For the analysis of the rational case with positions $x_j \in \R$ on the real line, we refer to \cite{LeSo-20}.] Since we are ultimately interested in taking the limit $N \to \infty$, we choose the coupling constants to be
$$
\kappa = N \quad \mbox{and} \quad a = \frac{1}{2N}
$$
from now on.

 %For a brief digression on the link between $H_{\mathrm{CM}}$ and the quantum HS model, the interested reader may consult Appendix ?? below.

Let us now explain the connection between the Hamiltonian system generated by $H_{\mathrm{CM}}^{(\mu)}$ and the half-wave maps equation (HWM). Recalling that we consider the trigonometric case, it will be convenient to identify the one-dimensional torus 
$$
\T \simeq \Ss = \{ z \in \C : |z| = 1\}
$$
with the one-dimensional unit sphere $\Ss$ embedded in the complex plane $\C$. Next, we consider initial positions $x_k = \frac{2 \pi k}{N}$ which are placed {\em equidistantly} on $\T \simeq \Ss$. That is, we consider the points 
$$
 z_k := e^{2 \pi \im k /N} \in \Ss \quad \mbox{with $k = 1, \ldots, N$}
$$ 
given by the $N$-th roots of units. Now, by formally taking the infinite-mass limit $\mu \to +\infty$, we obtain the reduced Hamiltonian 
\be
H^{\mathrm{(\infty)}}_{\mathrm{CM}} =  \frac{1}{4N}  \sum_{j \neq k}^{N} \frac{\Sb_j \cdot \Sb_k}{\sin^2 (\pi( j-k)/N)} = \frac{1}{N} \sum_{j \neq k}^N \frac{\Sb_j \cdot \Sb_k}{|z_j - z_k|^2},
\ee 
which only involves the spin variables. Physically speaking, this means that the particle positions $z_k=e^{2 \pi i k/N} \in \Ss$ are forced to be `frozen' in an equidistant configuration on $\Ss$ (given by the $N$-th roots of unity) with the classical spins $\Sb_k \in \Ss^2$ attached to each site $z_k \in \Ss$. The equation of motions for the spins that are generated by $H^{\mathrm{(\infty)}}_{\mathrm{CM}}$ are found to be
\be \label{eq:spin_CM}
\boxed{\dot{\Sb}_k = \frac{1}{2N} \sum_{{j = 1 \atop j \neq k}}^N \frac{\Sb_k \times \Sb_j}{\sin^2 (\pi(k-j)/N)}}
\ee
with $k=1, \ldots, N$ and where $\times$ denotes vector product in $\R^3$. We notice that the right-hand side can be formally seen as a Riemann sum approaching as $N \to \infty$ the right-hand side of (HWM), which can be written as
$$
(\Sb \times |\nabla| \Sb)(x) = \Sb(x) \times \frac{1}{\pi} PV \int_0^{2 \pi} \frac{\Sb(x) - \Sb(y)}{\sin^2 (x-y)} \, dy = \frac{1}{\pi} PV \int_0^{2 \pi} \frac{\Sb(x) \times \Sb(y)}{\sin^2(x-y)} \, dy 
$$
Here we used the well-known singular integral expression $|\nabla|$ on $\T \simeq \Ss$ combined with the simple fact that $\Sb(x) \times \Sb(x)=0$. In the following subsections, we will  address the continuum limit with $N \to \infty$ leading to (HWM) in a rigorous fashion.

\subsection{Continuum Limit leading to Global Weak Solutions}
For convenience, we will consider the case that $N \in 2 \N + 1$ is an odd integer. Furthermore, we use $\Ss_N$ to denote the set of lattice points on the unit circle $\Ss$ given by the $N$-th roots of unity, i.\,e, we set
$$
\Ss_N = \{ z_0, z_1, \ldots, z_{N-1} \} \subset \Ss \quad \mbox{with} \quad z_k = e^{2 \pi i k/N}, 
$$
with the integer index $k=0, \ldots, N-1$. (The inessential change from $k=1, \ldots, N$ to $k=0, \ldots, N-1$ turns out to be  convenient for our analysis below.) Moreover, we use the following notation for the initial-value problem of the coupled system of ODEs given in \eqref{eq:spin_CM}:
\be \label{eq:S_disc_periodic}
\left \{ \begin{array}{l} \displaystyle
\pt_t \Sb_N = \Sb_N \times |\nabla|_N \Sb_N, \\
\Sb_N |_{t=0} = \Sb_{N,0} .
\end{array} \right .
\ee   
Here $\Sb_{N,0} : \Ss_N \to \Ss^2$ is some given initial condition and the operator $|\nabla|_N$ defined as
\be
(|\nabla|_N f)(z_k) = \frac{2}{N} \sum_{\ell = 0, \ell \neq k}^{N-1} \frac{f(z_k)-f(z_\ell)}{|z_k - z_\ell|^2} = \frac{1}{2N} \sum_{\ell = 0, \ell \neq k}^{N-1} \frac{f(z_k)-f(z_\ell)}{\sin^2(\pi(k-\ell)/N)} 
\ee
for functions $f$ defined on $\Ss_N$. By the standard Cauchy--Lipschitz theory and a-priori bounds, it is easy to see that \eqref{eq:S_disc_periodic} has a unique solution $\Sb_N : [0,\infty) \times \Ss_N \to \Ss^2$ which is $C^1$ in $t$; see Section \ref{sec:prelim} below.

In order to study the limit as $N \to \infty$, we need to introduce a suitable  interpolation procedure for $\Sb_N : [0,\infty) \times \Ss_N \to \Ss^2$ of the discrete problem \eqref{eq:S_disc_periodic} that yields a function on $\Ss$. Here we will take the \textbf{trigonometric interpolation} of the solution $\Ss_N$ to the discrete problem.  More precisely, given a function $f : \Ss_N \to \C$ with $N = 2n +1$, we define its trigonometric interpolation $\wt{f}$ to be the function $\wt{f} : \Ss \to \C$ given by
\be
\wt{f}(z) = \sum_{k=-n}^n c_k z^k \quad \mbox{with} \quad c_k = \frac{1}{N} \sum_{k=0}^{N-1} f(z_k) \overline{z}^k.
\ee 
It is a well-known fact that $\wt{f} : \Ss \to \C$ is the unique trigonometric polynomial of degree $n$ such that $\wt{f}(z_k) = f(z_k)$ for all $k =0, \ldots, N-1$; i.\,e., its values coincide with the lattice function $f: \Ss_n \to \C$ on the lattice points $z_k \in \Ss_N$. Moreover, it is easy to see that $\wt{f}$ is real-valued whenever $f$ is real-valued.

Our first main result shows that the trigonometric interpolations of the solution  $\Sb_N : [0,\infty) \times \Ss_N \to \Ss^2$ yields a weak global-in-time finite-energy solution of (HWM) in the limit $N \to \infty$. Before we state this convergence result, we recall what we mean by a weak solution of (HWM) with finite energy, i.\,e.~for initial data in $H^{1/2}(\Ss; \Ss^2)$.

\begin{definition}\label{eq:def_weak_sol}
Let $\Sb_0 \in H^{\frac 1 2}(\Ss;\Ss^2)$ and $T>0$. We say that $\Sb : [0,T] \times \Ss \to \Ss^2$ is a \textbf{weak solution} of (HWM) if the following properties hold:
\begin{enumerate}
\item[$(i)$] $\Sb \in L^\infty([0,T]; H^{\frac 1 2}(\Ss))$ and $\pt_t \Sb \in L^\infty([0,T]; H^{-1/2}(\Ss))$.
\item[$(ii)$] For every $\bm{\phi} \in H^{1/2}(\Ss; \R^3)$, we have
$$
\langle \bm{\phi}, \pt_t \Sb(t) \rangle = \langle |\nabla|^{1/2} \Sb(t), |\nabla|^{1/2} (\bm{\phi} \times \Sb(t)) \rangle \quad \mbox{for a.\,e.~$t \in [0,T]$}.
$$
\item[$(iii)$] $\Sb(0,x) = \Sb_0(x)$ for a.\,e.~$x \in \Ss$.
\end{enumerate}
\end{definition}

\begin{remarks*}
1) The condition $(i)$ implies that $\Sb \in C^{1/2}([0,T]; L^2(\Ss))$ by interpolation. Hence the initial condition stated in $(iii)$ makes sense as the strong limit $\lim_{t \to 0^+} \Sb(t) = \Sb_0$ in $L^2(\Ss)$.

2) As will see below, the existence of global weak solutions for (HWM) (i.\,e., for arbitrary large $T>0)$ follows from a-priori bounds and compactness arguments. However, the uniqueness of weak solutions is a major open problem as singularity formation of smooth solutions may occur.

3) Below we will also consider strong local-in-time solutions of (HWM) with higher regularity and obtain a quantitative estimate on the convergence. 
\end{remarks*}

We are now ready to state the first main result of this paper.

\begin{theorem} \label{thm:main_periodic}
Suppose  $\{ \Sb_{0,N}  : \Ss_N \to \Ss^2 \}$ is a family of initial data for the discrete equation \eqref{eq:S_disc_periodic} with $N \in 2 \N + 1$ lattice sites and assume that 
$$
\sup_N \| \wt{\Sb}_{0,N} \|_{H^{1/2}(\Ss)} < +\infty
$$ 
holds for its corresponding trigonometric interpolations on $\Ss$. 

Then, for any $T> 0$ and $\eps > 0$, the functions $\{ \wt{\Sb}_N(t) \}_{t \in [0,T]}$ solving \eqref{eq:S_disc_periodic} converge (up to a subsequence) to some $\Sb$ such that
$$
\mbox{$\wt{\Sb}_N \weakto \Sb$ weakly-$*$ in $L^\infty\big([0,T]; H^{1/2}(\Ss)\big) \cap W^{1,\infty}\big([0,T]; H^{-1/2-\eps}(\Ss)\big)$ as $N \to \infty$,}
$$
and where  $\Sb : [0,T] \times \Ss \to \Ss^2$ is a weak solution of (HWM). In addition, we have the strong convergence
$$
\| \wt{\Sb}_N(t) - \Sb(t) \|_{L^2([0,T]; H^{1/2-\eps}(\Ss))} \to 0 \quad \mbox{as} \quad N \to \infty.
$$
\end{theorem}

\begin{remarks*}
 1) Notice that the trigonometric interpolation $\wt{\Sb}_N : [0,T] \times \Ss \to \R^3$ are in general not $\Ss^2$-valued. However, we show that the weak limit does indeed satisfy  $\Sb(t,x) \in \Ss^2$ for a.\,e.~$(t,x) \in [0,T] \times \Ss$. This is a nontrivial fact proven below.
 
 2) A simple way to construct initial data $\{ \Sb_{0,N} : \Ss_N \to \Ss^2 \}$ satisfying the uniform boundedness assumption in Theorem \ref{thm:main_periodic} consists in taking a sufficiently regular function $\Sb_0 : \Ss \to \Ss^2$ and choosing the sampling $\Sb_{N,0} : \Ss_N \to \Ss^2$ defined by $\Sb_{N,0}(z_k) = \Sb_0(z_k)$. For instance, it is easy so show that $\| \wt{\Sb}_{N,0} \|_{H^{1/2}} \leq C \| \Sb_0 \|_{H^1}$ with some constant $C > 0$ independent of $N$.  Thus we obtain initial data $\{ \Sb_{0,N} : \Ss_N \to \Ss^2 \}$ that satisfy the required assumption.
\end{remarks*}

\subsection{Convergence and Uniqueness for Higher Regularity and Short Times}

Note that Theorem~\ref{thm:main_periodic} only ensures convergence to a global-in-time weak solution of (HWM), even if the family of initial data $\{ \Sb_{N, 0} \}$ arises from sampling a smooth data $\Sb_0 : \Ss \to \Ss^2$; see also the remark above. However, if $\Sb_0 : \Ss \to \Ss^2$ is sufficiently regular, we know local-in-time existence and uniqueness of strong solutions for (HWM); see, e.\,g., \cite{PuGu-13} (or Appendix \ref{sec:uniqueness}). We will now establish that for such data taking the continuum limit yields these strong solutions on a sufficiently short time interval. We also prove the following convergence estimate with respect to $N$.

\begin{theorem}\label{thm:conv_strg}
Let $\Sb_0 \in H^{5/2}(\Ss;\Ss^2)$ and let $\Sb_{N,0}$ be its sampling $\Sb_{N,0}(z_k) = \Sb_0(z_k )$ with $z_k=e^{2 \pi k \im/N}$ and $N \in 2 \N +1$. Then the limiting solution  obtained in Theorem \ref{thm:main_periodic} coincides up to some time $T_*=T_*(\norm{\Sb^{(0)}}_{H^{5/2}})>0$ with the unique strong solution $\Sb \in C^0([0,T_*); H^{5/2})$ of (HWM) with initial data $\Sb(0) = \Sb_0$.

Furthermore, for any compact interval $I \subset [0,T_*)$, we have
$$
		\sup_{t \in I} \norm{\wt{\Sb}_N(t)- \Sb(t)}_{H^{1/2}}\le \frac{C}{N}
$$
with some constant $C = C(I, \| \Sb_0 \|_{H^{5/2}}) > 0$.
\end{theorem}

\begin{remarks*}
1) The convergence estimate follows from a Gr\"{o}nwall argument; see the proof below for more details.

2) The above result can be generalized to solutions in $H^\sigma$ with $\sigma \geq 5/2$. For simplicity, we only treat the case $\sigma=5/2$ here.
\end{remarks*}

\subsection{Comments on the Proofs}

In the seminal work \cite{SuSuBa-86} on the continuum limit for the classical Heisenberg model, Sulem et al.~derive the {\em Schr\"odinger maps} 
\be \label{eq:SM}
\pt_t \Sb = \Sb \times \Delta \Sb
\ee
for $\Sb :[0,T) \times \R^d \to \Ss^2$ from the discrete spin Hamiltonian involving only nearest-neighbor interactions on a lattice $h \Z^d$ with mesh size $h \to 0^+$. An essential fact used in \cite{SuSuBa-86} is that the Sch\"odinger maps equation \eqref{eq:SM} can be  neatly cast into divergence form, namely 
$$
\pt_t \Sb = \sum_{j=1}^d \pt_{x_j} \left ( \Sb \times \pt_{x_j} \Sb  \right ).
$$
Moreover, this divergence form property also holds in analogous fashion for the discrete equation using finite difference operators. 

By contrast, no such elementary manipulation seems to be known for (HWM) and its discrete counterpart generated by $H^{(\infty)}_{\mathrm{CM}}$. As a consequence, the methods in \cite{SuSuBa-86} do not seem to be suitable for the present problem. To overcome this difficulty, we follow an approach inspired by \cite{KiLeSt-13}, where continuum limits of long-range lattice NLS were studied. However, the resulting limiting equation (a so-called fractional NLS) derived in \cite{KiLeSt-13} is of semilinear nature, whereas (HWM) is {\em quasilinear} leading to more analytic challenges. A particular feature that we exploit in the proof is a specific {\em cancellation property} which is essential when controlling the continuum limit $N \to \infty$. 

Finally, we remark that in our companion work \cite{LeSo-20} we will further address the rigorous derivation of (HWM) on the real line $\R$ and we will also consider more general interaction potentials $V(x)=1/|x|^\alpha$ with $\alpha >0$ in \eqref{eq:H_CM_def}.

\subsection*{Acknowledgments}
Both authors were supported by the Swiss National Science Foundations (SNF) through Grant No.~20021-169464.

\section{Preliminaries}

\label{sec:prelim}

In this section, we collect some preliminary results.

\subsection{Calogero--Moser Spin Model on $\Ss_N$}

Recall that $\Ss_N = \{ z_0, \ldots, z_{N-1} \} \subset \Ss$ denotes the set of lattice points $z_k = e^{2 \pi i k/n}$ with $k=0, \ldots, N-1$, where $N \geq 2$ is an integer. At this point, we do not force $N$ to be odd. Furthermore, we recall the following Hamiltonian 
\be\label{eq:Hamiltonian_HWM}
H_{\mathrm{CM}}^{(\infty)}(\Sb_N) = \frac{1}{N} \sum_{k \neq \ell}^N \frac{|\Sb_N(z_k)-\Sb_N(z_\ell)|^2}{|z_k - z_\ell|^2} = \frac{1}{4N} \sum_{k \neq \ell}^N  \frac{|\Sb_N(z_k)-\Sb_N(z_\ell)|^2}{\sin^2(\pi(k-\ell)/N)} 
\ee
for spin configurations $\Sb_N : \Ss_N \to \Ss^2$. Denoting $\Sb_{N,k} = (S_N^1(z_k), S_N^2(z_k), S_N^3(z_k)) \in \Ss^2$, the canonical Poisson brackets for the spin variables are given by 
\be
\{S_{N,k}^a, S_{N,l}^b \} = \delta_{kl} \eps_{abc} S_{N,k}^c
\ee
for $k,l =0, \ldots, N-1$ and $a,b=1,2,3$, where $\eps_{abc}$ is the anti-symmetric Levi-Civit\`a symbol. An elementary calculation shows that the Hamiltonian equations of motion generated by $H_{\mathrm{CM}}^{(\infty)}$ are given by
\be \label{eq:Spin_per_class}
\left \{ \begin{array}{l} \displaystyle
\pt_t \Sb_N = \Sb_N \times |\nabla|_N \Sb_N, \\
\Sb_N |_{t=0} = \Sb_{N,0},
\end{array} \right .
\ee   
where $\Sb_{N,0} : \Ss_N \to \Ss^2$ is some given initial condition and we recall that the linear operator $|\nabla|_N$ acting on functions $f: \Ss_N \to \C$ is defined follows: 
\be \label{def:nabla_N}
(|\nabla|_N f)(z_k) = \frac{2}{N} \sum_{\ell = 0, \ell \neq k}^{N-1} \frac{f(z_k)-f(z_\ell)}{|z_k - z_\ell|^2} = \frac{1}{2N} \sum_{\ell = 0, \ell \neq k}^{N-1} \frac{f(z_k)-f(z_\ell)}{\sin^2(\pi(k-\ell)/N)} .
\ee 
About the initial-value problem \eqref{eq:Spin_per_class} we record the following facts.

\begin{lemma} \label{lem:GWP_S_per_class}
For any initial condition $\Sb_{N,0} : \Ss_N \to \Ss^2$, there exists a unique global-in-time solution $\Sb_N: [0, \infty) \times \Ss_N \to \Ss^2$ of \eqref{eq:Spin_per_class}. Moreover, we have conservation law
$$
H_{\mathrm{CM}}^{(\infty)}(\Sb_N(t)) = H_{\mathrm{CM}}^{(\infty)}(\Sb_0) \quad \mbox{for all $t \geq 0$}.
$$ 
\end{lemma}

\begin{proof}
This is straightforward to show. Indeed, local-in-time existence and uniqueness of solutions for \eqref{eq:Spin_per_class} follows from a standard Cauchy--Lipschitz argument for ODEs.  From \eqref{eq:Spin_per_class} and the fact that $|\Sb_N(0, \cdot)|^2 = 1$ (by the initial condition) we readily deduce that $|\Sb_N(t,\cdot)|^2 = 1$ as long as the solution exists, which is an a-priori bound allowing us to extend the solutions to all times $t \geq 0$. Finally, as an elementary consequence of the Hamiltonian  nature, we readily check that the energy is conserved.
\end{proof}

Next, we collect some basic properties of the operator $|\nabla|_N$, which can be seen as an approximation of $|\nabla|$ on the unit circle $\Ss$. With some slight (but obvious) abuse of notation, we can regard $|\nabla|_N$ as a linear map from $\C^N$ to $\C^N$. Clearly, we have that $|\nabla|_N = |\nabla|_N^*$ is self-adjoint (see also below). Its eigenvalues are found to be as follows.

\begin{lemma} \label{lem:EV_circle}
Let $N \in \N$ be given (and not necessarily odd). Then the eigenvalues of $|\nabla|_N : \C^N \to \C^N$ are given by
$$
\mu_k = k \left ( 1 - \frac{k}{N} \right ) \quad \mbox{with $k=0, \ldots, N-1$},
$$
with the corresponding normalized eigenvectors
$$
\mathbf{V}^{(k)} = \frac{1}{\sqrt{N}} \left (1, z^k, z^{2k}, \ldots, z^{k(N-1)} \right ) \quad \mbox{with $z=e^{2 \pi \im/N}$}.
$$
\end{lemma}

\begin{remark}
If $N=2n + 1$ is odd, then the eigenvalues $\{ \mu_k \}$ can  be listed as
$$
\mu_k = |k| \left (1 - \frac{|k|}{N} \right ) \quad \mbox{with $k=-n, \ldots, 0, \ldots, n$}.
$$
In particular, every non-zero eigenvalue $\mu_k \neq 0$ is doubly degenerate, whereas $\mu_0=0$ is a simple eigenvalue. Notice also that the initial-value problem \eqref{eq:Spin_per_class} has \textbf{stationary solutions} given by 
$$
\Sb_N = \mathbf{V}^{(k)}
$$
for each $k$. In some sense, these solutions correspond to discretized versions of \textbf{half-harmonic maps} which are stationary solutions for (HWM); see, e.\,g., \cite{LeSc-18}. 
\end{remark}

\begin{proof}
Let $C=[C_{k,j}]$ denote the matrix for $|\nabla|_N$ in the canonical basis for $\C^N$. We readily check $C$ is a \textbf{circulant matrix}, i.\,e., it holds
$$
C_{k,j} = c_{(j-k) \, \mathrm{mod} \, N}
$$
with some vector $\mathbf{c} = (c_0, \ldots, c_{N-1}) \in \C^N$. We find
$$
c_k = \begin{dcases*}  \frac{2}{N} \sum_{l=1}^{N-1} \frac{1}{|1-z_l|^2} & for $k=0$, \\
 -\frac{2}{N} \frac{1}{|1-z_k|^2} & for $k=1, \ldots, N-1$. \end{dcases*}
$$ 
Thus the eigenvalues of the matrix $C$ are found to be 
$$
\mu_j = \sum_{k=0}^{N-1} c_k \overline{z}_k^j = \sum_{k=0}^{N-1} c_k z_k^j \quad \mbox{for $j=0, \ldots, N-1$},
$$
where in the last line we used that $c_k \in \R$ and $\mu_j \in \R$ is real (since $C^* = C$ holds). Hence
$$
\mu_j = c_0 - \frac{2}{N} \sum_{k=1}^{N-1} \frac{z_k^j}{|1-z_k|^2} = \frac{1}{N} \sum_{k=1}^{N-1} \frac{1-z_k^j}{|1- z_k|^2}.
$$
To evaluate this sum, we recall from \cite{Ha-88} the following formula: 
$$
 \sum_{k=1}^{N-1} \frac{z^{kJ}}{(1-z^k)(1-z^{-k})} = \frac{1}{12} ( N^2-1) -\frac{1}{2} J (N-J),
$$
where $z=e^{2\pi \im /N}$ and  $J =0, \ldots, N$. Since $|1-z^k|^2 = (1-z^k)(1-z^{-k})$ for $z \in \Ss$, we readily get
$$
\mu_k =  k \left ( 1- \frac{k}{N} \right ) \quad \mbox{for $k=0, \ldots, N-1$}.
$$ 
\end{proof}

For further use below, we introduce the following notation. As previously noticed, we can identify the set $\{ f : \Ss_N \to \C \}$ of complex-valued maps on $\Ss_N$ with the Hilbert space 
$$
\ell^2(\Ss_N) \simeq \C^{N} = \C^{2n+1},
$$
which we equip with the scalar product
$$
\langle f, g \rangle_{\ell^2(\Ss_N)} = \frac{1}{N} \sum_{z_k \in \Ss_N} \overline{f(z_k)} g(z_k)  \quad \mbox{for $f,g \in \ell^2(\Ss_N)$}.
$$
Recall that $|\nabla|_N$ is a self-adjoint operator $\ell^2(\Ss_N)$ with nonnegative eigenvalues $\mu_k \geq 0$. For any real number $s \geq 0$, we define the power $|\nabla|_N^s$ by spectral calculus so that accordingly the eigenvalues are given by $\mu_j^s$. Likewise, we define $\langle |\nabla|_N \rangle^s = (\mathds{1} + |\nabla|_N^2)^{s/2}$ by spectral calculus, whence the corresponding eigenvalues are given by $(1+\mu_j^2)^{s/2}$.

\subsection{Trigonometric Interpolation and Flow on $\mathrm{Pol}_n$}

For $n \in \N$, we use
\be
\mathrm{Pol}_n = \left \{ \sum_{k=-n}^n c_k z^k : c_k \in \C \right \}
\ee
to denote the set of trigonometric polynomials in $z \in \Ss$ of degree at most $n$. Now let $N = 2n + 1$ be an odd integer in what follows. For any lattice function $f \in \ell^2(\Ss_N)$, we define its \textbf{trigonometric interpolation} as the (unique) trigonometric polynomial $\wt{f} \in \mathrm{Pol}_n$ such that its values coincide on the points $\Ss_N$, i.\,e., we have $\wt{f}(e^{2 \pi \im k /N}) = f_k$ for $k=0, \ldots, N-1$. It is a standard fact that $\wt{f} \in \mathrm{Pol}_n$ is given by the formula
\be\label{eq:def_f_tilde}
\wt{f}(z) = \sum_{k=-n}^n c_k z^k \quad \mbox{where}  \quad c_k = \frac{1}{N} \sum_{k=0}^{N-1} f_k \overline{z}^k.
\ee 
Besides the trigonometric interpolation map $\ell^2(\Ss_n) \to \mathrm{Pol}_n$ with $f \mapsto \wt{f}$, we also introduce the closely related map
\be \label{def:cI_N}
\cI_N(u) = \sum_{k=-n}^n c_k(u) z^k \quad \mbox{with} \quad c_k(u) = \sum_{j \in \mathbb{Z}} \widehat{u}_{k+jN} ,
\ee
which is well-defined for any function $u = \sum_{k \in \Z} \widehat{u}_k z^k$ with Fourier coefficients $(\widehat{u}_k)_{k \in \Z} \in \ell^1(\Z)$. Thus we have map $\cI_N : \widehat{\ell^1(\Z)} \to \mathrm{Pol}_n$ with the obvious property that $\cI_N(u) = u$ whenever $u \in \mathrm{Pol}_n$.  

For further use, we also introduce the following orthogonal projections acting on $u = \sum_{k \in \mathbb{Z}} \widehat{u}_k z^k \in L^2(\Ss)$ given by
\be
P_{n}(u) = \sum_{|k| \leq n} \widehat{u}_k z^k, \quad \Pi_+(u) = \sum_{k \geq 0} \widehat{u}_k z^k .
\ee
That is, $P_n$ denotes the projection onto the space $\mathrm{Pol}_n$ of trigonometric polynomials of degree at most $n$, whereas $\Pi_+$ is the Szeg\H{o} projection onto $L^2(\Ss)$-functions with nonnegative frequencies. Likewise, we set $P_n^\perp = \mathds{1}-P_n$ and $\Pi_- = \mathds{1}-\Pi_+$ for projections onto the corresponding orthogonal complements. Furthermore, we introduce the projections:
\be
P_{n,+}^\perp(u) = P_{n}^\perp \Pi_+(u) = \sum_{k > n} \widehat{u}_k z^k, \quad P_{n,-}^\perp(u) = P_{n}^\perp \Pi_-(u) = \sum_{k < -n} \widehat{u}_k z^k.
\ee
We can now list some basic properties of the mappings $f\mapsto \wt{f}$ and $u \mapsto \cI_N(u)$ that will be needed further below. %We recall $\norm{f}_{\ell^2(\Ss_N)}^2:=\frac{1}{N}\sum_{z\in\Ss_N}|f(z)|^2$.

\begin{proposition} \label{prop:trig_inter}
Suppose $N = 2n + 1$ with $n \in \N$ and let $f, g : \Ss_N \to \C$ be given. Then the following properties hold true.
\begin{enumerate}
\item[$(i)$]  $\displaystyle \| \wt{f} \|_{L^2(\Ss)} = \| f\|_{\ell^2(\Ss_N)}$.
\item[$(ii)$] For all $s \geq 0$, we have 
$$
 \| |\nabla|^s \wt{f} \|_{L^2(\Ss)} \sim \||\nabla|_N^s f \|_{\ell^2(\Ss_N)}, \quad \| \wt{f} \|_{H^s(\Ss)} \sim \| \langle |\nabla|_N \rangle^s f \|_{\ell^2(\Ss_N)},
$$
where the constants are independent of $N$.
\item[$(iii)$] We have the product formula
$$
\wt{(fg)} = \cI_N(\wt{f} \, \wt{g}) = \big [ P_n+\overline{z}^NP_{n,+}^{\perp}+z^NP_{n,-}^{\perp}\big ] \big (\wt{f} \wt{g} \big ),
$$
where $P_n, P_{n,+}^\perp$ and $P_{n,-}^\perp$ denote the projection operators introduced above.
\end{enumerate}
\end{proposition}

\begin{proof}
%To be inserted.
\noindent Using \eqref{eq:def_f_tilde}, we get
$
 \sum_{-n\le k\le n}|c_k|^2=\frac{1}{N}\sum_{k=0}^{N-1}|f_k|^2,
$
which proves (i). From Lemma~\ref{lem:EV_circle}, we have $\frac{|k|}{2}\le |\mu_k|\le |k|$ for all $-n\le k\le n$ whence (ii) follows. By construction, $\wt{(fg)}$ and $\wt{f} \, \wt{g}$ take the same values on $\Ss_N$, so we have $\wt{(fg)} = \cI_N(\wt{f} \, \wt{g})$. We deduce (iii) by observing that $\wt{f} \, \wt{g}$ is in $\mathrm{Pol}_{2n}$.
\end{proof}

Next, we state useful properties of the trigonometric polynomial $\cI_N(f)$ approximating a sufficiently regular function $f$ on $\Ss$. We have the following result.

\begin{lemma}\label{lem:approx}
	Let $N=2n+1$ be odd. Suppose $\eps>0$ and $f\in H^{1/2+\eps}(\Ss)$ and $s\in [0, 1/2+\eps)$. Then for some constant $C(\eps) > 0$:
	\begin{itemize}
		\item[$(i)$] $\norm{\cI_N(f)}_{H^{1/2+\eps}}\le C(\eps)\norm{f}_{H^{1/2+\eps}}$, and
		\item[$(ii)$] $\norm{f-\cI_N(f)}_{H^{s}}\le \frac{C(\eps)}{N^{1+\eps-s}}\norm{f}_{H^{1/2+\eps}}$.
	\end{itemize}
\end{lemma}

\begin{proof}
See Appendix \ref{sec:aux} below. 
\end{proof}

We finally reformulate the equation of motion \eqref{eq:Spin_per_class} in terms of its trigonometric interpolation. To this end, we define the linear operator
\be 
D_N : \mathrm{Pol}_n \to \mathrm{Pol}_n
\ee 
to denote the linear operator that is conjugated to $|\nabla|_N : \ell^2(\Ss_N) \to \ell^2(\Ss_N)$ by using the trigonometric interpolation map $f \mapsto \wt{f}$, which is an isometry between $\ell^2(\Ss_n)$ and $\mathrm{Pol}_n$ by Proposition \ref{prop:trig_inter}. That is, we set
\be 
D_N  u := \wt{(|\nabla|_N f_u)} \quad \mbox{for $u \in \mathrm{Pol}_n$},
\ee
where $f_u\in \ell^2(\Ss_N)$ is the $\Ss_N$-sampling of $u$. %where $f_u$ is the unique element in $\ell^2(\Ss_N)$ such that $\wt{f_u} = u$. 
In view of Lemma \ref{lem:EV_circle}, we see that $z^j$ and $z^{-j}$ with $j=0, \ldots, n$ are eigenfunctions of $D_N$ with corresponding eigenvalues $\mu_j = |j|(1-|j|/N)$.

We have the following result.

\begin{proposition}[Flow on $\mathrm{Pol}_n$] \label{prop:trig_flow}
Suppose $N = 2n+1$ with $n \in \N$ and assume $\Sb_N : [0, \infty) \times \Ss_N \to \Ss^2$ is a global-in-time solution of \eqref{eq:Spin_per_class} as given by Lemma \ref{lem:GWP_S_per_class}. Then its trigonometric interpolation $\wt{\Sb}_N : [0,\infty) \times \Ss \to \R^3$ solves the initial-value problem
\be \label{eq:Sb_trig_flow}
\left \{ \begin{array}{l} \pt_t \wt{\Sb}_N = \cI_N \big ( \wt{\Sb}_N \times D_N \wt{\Sb}_N \big), \\
\wt{\Sb}_N(0) = \wt{\Sb}_{N,0}. \end{array} \right .
\ee
\end{proposition}

\begin{remark*}
Recall that the normalization condition $\Sb_N(t) \in \Ss^2$ is {\em not} preserved by trigonometric interpolation. In fact, we only  have the general upper bound $\| \wt{\Sb}_N \|_{L^\infty(\Ss)} \lesssim N$ which blows up as $N \to \infty$. As a consequence, the proof that $\wt{\Sb}_N$ will converge (up to a subsequence) to a weak solution of half-wave maps equation with values in $\Ss^2$ will require some careful analysis done below.  
\end{remark*}

\begin{proof}
This follows from the product formula stated in Proposition \ref{prop:trig_inter}. 
\end{proof}

\subsection{Piecewise Constant Interpolations}\label{sec:comparison_shanon_piec_wise_R}

For later use, we briefly discuss a relation between the trigonometric interpolation and the piecewise constant interpolation of a given map on the set $\Ss_N$ of lattice points.  For $\Sb_N : \Ss_N \to \Ss^2$, we define its \textbf{piecewise constant interpolation} $\mathbf{\Sigma}_N : \Ss \to \R^3$ to be
	\begin{equation}\label{eq:piece_wise_cst_extension}
		\mathbf{\Sigma}_N(z):=\sum_{j=0}^{N-1}\Sb_N(z_j) \,\mathbf{1}_{I_j}(z) 
	\end{equation}
	with $z_j=e^{2\pi j \im /N}$ as usual and the characteristic functions are defined as
	\be
	\mathbf{1}_{I_j}(z) = \begin{dcases*} 1 & if $\arg z -\arg z_j \in [-\frac{\pi}{N}, \frac{\pi}{N})$, \\ 0 & else. \end{dcases*}
	\ee
We have the following convergence result.
 
	\begin{lemma}\label{lem:compare_piece_cst_T}
	 Suppose that $\{ \Sb_N : \Ss_N \to \Ss^2\}$ is a family with $N \in 2\N +1$ and let $\wt{\Sb}_N$ and $\mathbf{\Sigma}_N$ denote the corresponding trigonometric interpolations and piecewise constant interpolations, respectively. Moreover, we assume that $\sup_N\norm{\wt{\Sb}_N}_{H^s(\Ss)} < +\infty$ holds for some $s > 0$. Then we have the strong convergence
	 $$
	 \lim_{N\to+\infty}\norm{\wt{\Sb}_N-\mathbf{\Sigma}_N}_{L^2(\Ss)}=0.
	 $$
	 %converges to $0$ as $N\to+\infty$.
	\end{lemma}

\begin{proof}
See Appendix \ref{sec:aux} below.
\end{proof}

We note in passing that a similar result holds for the piecewise linear extension, but we do not need it.

\section{Convergence to Weak Solution: Proof of Theorem \ref{thm:main_periodic}}

This section is devoted to the proof of Theorem \ref{thm:main_periodic} which establishes convergence to a weak solution of (HMW) in the limit $N \to \infty$. We begin with some preliminary results that will be needed below.

\subsection{Error Estimates and A-Priori Bounds}

Let $\wt{\Sb}_N : [0, \infty) \times \Ss \to \R^3$ be the trigonometric interpolation as given by Proposition \ref{prop:trig_flow} above. We define the error term
\be \label{def:E_error}
\EE(\wt{\Sb}_N(t)) = \cI_N( \wt{\Sb}_N(t) \times D_N \wt{\Sb}_N(t)) - (\wt{\Sb}_N(t) \times D_N \wt{\Sb}_N(t)).
\ee
We will now show the following key estimate.

\begin{lemma} \label{lem:error}
For every $T >0$ and $\eps > 0$, we have
$$
\mbox{$\| \EE(\wt{\Sb}_N(t)) \|_{H^{-1/2-\eps}(\Ss)} \to 0$ as $N \to \infty$ uniformly in $t \in [0,T]$.} 
$$
\end{lemma}

\begin{remark*}
The proof below will also show that $\| \EE(\wt{\Sb}_N(t)) \|_{H^{-1/2}(\Ss)}$ is uniformly bounded in $N$ and $t \in [0,T]$. 
\end{remark*}

\begin{proof}
For notational convenience, we omit the time variable $t$  and we also write $\EE_N$ to denote $\EE(\wt{\Sb}_N)$. Suppose that $\eps > 0$ and $T >0$ are given. Recall that $N= 2n +1$ with some $n \in \N$.
  
We recall \eqref{def:cI_N} to compute the Fourier coefficients of $\cI_N(\wt{\Sb}_N \times D_N\wt{\Sb}_N)$. Note that, by construction, we have $\wt{\Sb}_N,D_N\wt{\Sb}_N \in \mathrm{Pol}_n$. Furthermore, the product formula in Proposition \ref{prop:trig_inter} gives us	
\be
	\EE_N=-\big((\overline{z}^{N}-1)P_{n,+}^{\perp}+(z^N-1)P_{n,-}^{\perp}\big)(D_N\wt{\Sb}_N \times \wt{\Sb}_N ) .
\ee
Let us denote by $\wh{\Cb}_j\in \C^3$ the Fourier coefficients of $(D_N\wt{\Sb}_N \times \wt{\Sb}_N) \in \mathrm{Pol}_{2n}$. By Parseval's identity, we have
\be\label{eq:norm_error}
\| \EE_N \|_{H^{-1/2-\eps}(\Ss)}^2 = \sum_{k=1}^n \left ( \frac{1}{\langle n+k\rangle^{1+2 \eps}} + \frac{1}{\langle n+1-k \rangle^{1+2 \eps}} \right ) \left ( |\wh{\Cb}_{n+k}|^2 + |\wh{\Cb}_{-n-k}|^2 \right ).
\ee
Next, we will estimate the Fourier coefficients $\wh{\Cb}_{n+k}$, which will be sufficient for our purpose, since $\overline{\wh{\Cb}_{-n-k}} = \wh{\Cb}_{n+k}$ by the real-valuedness of $D_N \wt{\Sb}_N \times \wt{\Sb}_N$. Let us write $\wh{\Sb}_j\in \C^3$ (with $|j|\le n$) for the Fourier modes of $\wt{\Sb}_N$, and we recall that the eigenvalues of $D_N$ are  given by $\mu_j:=|j|\big(1-\tfrac{|j|}{N}\big)$ (with eigenvectors $z^j$ and $z^{-j}$ in $\mathrm{Pol}_n$).  For $1\le k\le n$, we thus find
\be\label{eq:cancellations}
\wh{\Cb}_{n+k} =\sum_{j=k}^n \mu_{j}\wh{\Sb}_j \times \wh{\Sb}_{n+k-j} = \sum_{j=\bcei{\tfrac{n+k}{2}}}^n (\mu_{j}-\mu_{n+k-j})\wh{\Sb}_j\times  \wh{\Sb}_{n+k-j},
\ee
where in the last step we used an important {\em cancellation property} due to the skew-symmetry of the vector product. From the explicit formula for the eigenvalues $\mu_j$ (see Lemma \ref{lem:EV_circle}) we observe the bounds
\be
0 \leq \mu_j - \mu_{n+k-j} = (2j-n-k) \left (1+ \frac{n+k}{N} \right ) \leq 2 \left (2j - (n+k) \right ) \leq 2 j
\ee
for $1 \leq j \leq k \leq n$. By using these bounds and applying the Cauchy--Schwarz inequality,
\begin{align*}
& | \wt{\Cb}_{n+k} |^2  \lesssim \left ( \sum_{j=\bcei{\tfrac{n+k}{2}}}^n j |\wh{\Sb}_j|^2 \right ) \left ( \sum_{j=\bcei{\tfrac{n+k}{2}}}^n ( (2j-(n+k)) |\wh{\Sb}_{n+k-j}|^2 \right ) \\
& \lesssim \|  \wt{\Sb}_N \|_{H^{1/2}(\Ss)}^2  \left ( \sum_{j=\bcei{\tfrac{n+k}{2}}}^n ( (2j-(n+k)) |\wh{\Sb}_{n+k-j}|^2 \right ) .
\end{align*}
Hence it follows 
\begin{align}\label{eq:estimate_error}
& \sum_{k=1}^n \frac{| \wt{\Cb}_{n+k} |^2}{\langle n+1-k \rangle^{1+2 \eps}}  \lesssim  \|  \wt{\Sb}_N \|_{H^{1/2}(\Ss)}^2  \left ( \sum_{k=1}^n \sum_{j=\bcei{\tfrac{n+k}{2}}}^n  \frac{(2j-(n+k))}{\langle n+1-k\rangle^{1+2 \eps}} |\wh{\Sb}_{n+k-j}|^2 \right ) \\
& \lesssim   \| \wt{\Sb}_N \|_{H^{1/2}(\Ss)}^2 \left ( \sum_{\ell=1}^{n-1} |\wh{\Sb}_\ell|^2 \underbrace{\sum_{j=\max(\ell,n+1-\ell)}^n \frac{j-\ell}{\langle 2n+1 - \ell -j \rangle^{1 + 2 \eps}}}_{=: F(\ell,n)} \right )\nonumber,
\end{align}
where we have made the substitution $\ell=n+k-j$ in the sum over $k$. We now claim that
\be \label{ineq:A_ell_n}
\frac{F(\ell,n)}{\ell} =o(1) \quad \mbox{uniformly in } \ell \mbox{ as }  n \to \infty.
\ee
Indeed, for $1 \leq \ell \leq \bflo{\tfrac{n+1}{2}}$, we notice
\begin{align*}
\frac{F(\ell,n)}{\ell} &= \frac{1}{\ell} \sum_{j=n+1-\ell}^n \frac{j-\ell}{\langle 2n+1-\ell-j \rangle^{1+2 \eps}} = \frac{1}{\ell} \sum_{m=n-\ell+1}^n \frac{2(n-\ell) +1-m}{\langle m \rangle^{1+2 \eps}} \\
& \leq  (n-\ell) \frac{1}{\ell} \int_{n-\ell}^n \frac{dy}{\langle y \rangle^{1+2 \eps}} \\
& \lesssim \frac{n-\ell}{\big\langle n-\ell \big\rangle^{1+2\eps}} =o(1) \quad  \mbox{uniformly as} \quad n \to \infty.
% & \lesssim \frac{1}{2 \eps} (n+1) \left ( \frac{1}{\langle n \rangle^{2\eps}} - \frac{1}{\langle n - \ell \rangle^{2 \eps}} \right ) =o(1) \quad \mbox{as} \quad n \to \infty.
 \end{align*}
Similarly, for $\frac{n+1}{2} < \ell \leq n-1$, we obtain that
\begin{align*}
\frac{F(\ell,n)}{\ell}  & = \sum_{j=\ell}^n \frac{j-\ell}{\langle 2n +1- \ell -j \rangle^{1 + 2 \eps}} = \frac{1}{\ell} \sum_{m=n-\ell+1}^{2(n-\ell)+1} \frac{2 (n-\ell) + 1 -m}{\langle m \rangle^{1 + 2 \eps}} \\
& \leq  (n-\ell)\frac{1}{\ell} \int_{n-\ell}^n  \frac{dy}{\langle y \rangle^{1 + 2 \eps}}.%\le \frac{ (n-\ell)}{\langle n-\ell \rangle^{1+ 2 \eps}} 
%= o(1) \quad \mbox{uniformly as} \quad n \to \infty.
\end{align*}
We now claim that the right-hand side above is also $o(1)$ uniformly as $n \to \infty$. To see this, we note
$$
\frac{n-\ell}{\ell} \int_{n-\ell}^n  \frac{dy}{\langle y \rangle^{1 + 2 \eps}}= \tfrac{1-u}{u}\big[G((1-u)n)-G(n)\big],
$$
where we set
$$
G(x) :=\int_x^{+\infty}\frac{\mathrm{d} y}{\langle y\rangle^{1+2\eps}} \quad \mbox{and} \quad u:=\frac{\ell}{n}\in \left [ \frac 1 2,1 \right ].
$$ 
Next, we observe 
\begin{align*}
    \tfrac{1-u}{u}\big[G((1-u)n)-G(n)\big] & \le 2(1-u)G((1-u)n)  \\
    &  \le 2\sup_{1/2 \le u\le 1}(1-u)G((1-u)n)\underset{n\to+\infty}{\longrightarrow}0,
\end{align*}
where the last step can be seen as follows. Let $\eta>0$ and take $A(\eta)$ such that $G(A(\eta))=\eta$. We have $(1-u_*)n=A(\eta)$
for $u_*=u_*(\eta,n)=1-A(\eta)/n$. Hence, for $1/2\le u\le u_*(\eta,n)$ we have $(1-u)G((1-u)n)\le F(A(\eta))\le \eta$, and for $u_*(n,\eta)<u<1$ we have
\[
 (1-u)G((1-u)n)\le (1-u_*(\eta,n))G(0)=\tfrac{A(\eta)}{n}G(0)\underset{n\to+\infty}{\longrightarrow}0.
\]

In summary, we conclude
\begin{align*}
\sum_{k=1}^n \frac{| \wt{\Cb}_{n+k} |^2}{\langle n+1-k \rangle^{1+2 \eps}} & \lesssim   \| \wt{\Sb}_N \|_{H^{1/2}(\Ss)}^2 \left ( \sum_{\ell=1}^{n-1} |\wh{\Sb}_\ell|^2 \ell \right ) \cdot o_{n\to \infty}(1) \\
& \lesssim   \| \wt{\Sb}_N \|_{H^{1/2}(\Ss)}^4  \cdot o_{n\to \infty}(1),
\end{align*}
where the analogous estimate follows for the sum $\sum_{k=1}^n \frac{| \wt{\Cb}_{-(n+k)} |^2}{\langle n+1-k \rangle^{1+2 \eps}}$ as remarked above. Since $\langle n+k \rangle^{-(1+2\eps)} \leq \langle n+1-k \rangle^{-(1+2 \eps)}$ for $1 \leq k \leq n$, we finally deduce that
\be
\| \EE_N \|_{H^{-1/2-\eps}(\Ss)}^2 \lesssim  \| \wt{\Sb}_N \|_{H^{1/2}(\Ss)}^4  \cdot o_{n\to \infty}(1) ,
\ee
which in combination with the a-priori bound $\| \wt{\Sb}_N \|_{H^{1/2}} \lesssim \|  |\nabla|_N^{1/2} \Sb_N \|_{\ell^{2}(\Ss_N)} \lesssim 1$ completes the proof of Lemma \ref{lem:error}.
\end{proof}

We next prove the following a-priori bound.

\begin{lemma} \label{lem:apriori}
Let $N= 2n +1$ with $n \in \N$, $\eps > 0$, and $\wt{\Sb}_N : [0,\infty) \times \Ss \to \R^3$ be as above. For any $\bm{\phi} \in H^{1/2+\eps}(\Ss; \R^3)$, we have
$$
\big | \langle \bm{\phi}, \wt{\Sb}_N(t) \times D_N \wt{\Sb}_N(t) \rangle_{L^2} \big | \leq C \| \bm{\phi} \|_{H^{1/2+\eps}(\Ss)} \quad \mbox{for all $t \geq 0$},
$$
with some constant $C>0$ independent of $N$ and $t \geq 0$.
\end{lemma}

\begin{proof}
We extend the operator $D_N : \mathrm{Pol}_n \to \mathrm{Pol}_n$ to all of $L^2(\Ss)$ by setting $D_N z^\ell = \tfrac{|\ell|}{2} z^\ell$ for $\ell \in \Z$ with $|\ell| > n$ (the factor $\tfrac{1}{2}$ is important to ensure \eqref{eq:bound_difference}). Thus $D_N$ is a nonnegative self-adjoint operator on $L^2(\Ss)$, with operator domain $H^1(\Ss)$, such that $D_N z^\ell = \mu_\ell z^\ell$ for $|\ell| \leq n$ and $D_N \ell = \frac{|\ell|}{2} z^{\ell}$ for $|\ell| > n$. Likewise, we define fractional powers $D_N^s$ with $s > 0$ by spectral calculus and we note the norm equivalence $\| D_N^s u \|_{L^2(\Ss)} \sim \| |\nabla|^s u \|_{L^2(\Ss)}$ with constants independent of $N$.

As usual, we omit the $t$ variable in what follows. We rewrite $\langle \bm{\phi}, \wt{\Sb}_N \times D_N \wt{\Sb}_N \rangle_{L^2}$ in ``divergence form'' by extracting $D^{1/2}_N$ from $D_N \wt{\Sb}_N$ while using the cyclicity of $\mathbf{a} \cdot (\mathbf{b} \times \mathbf{c})$ for $\mathbf{a}, \mathbf{b}, \mathbf{c} \in \R^3$ as well as $D_N=D_N^{1/2} D_N^{1/2}$. We thus get
\begin{align*}
\langle \bm{\phi}, \wt{\Sb}_N \times D_N \wt{\Sb}_N \rangle_{L^2} & = \langle D^{1/2}_N \bm{\phi}, \wt{\Sb}_N \times D_N^{1/2} \wt{\Sb}_N \rangle_{L^2} -  \langle D^{1/2}_N \bm{\phi}, \wt{\Sb}_N \times D_N^{1/2} \wt{\Sb}_N \rangle_{L^2} \\
& \quad + \langle \bm{\phi}, \wt{\Sb}_N \times D_N \wt{\Sb}_N \rangle_{L^2} \\
& = \langle D_N^{1/2} \bm{\phi}, \wt{\Sb}_N \times D_N^{1/2} \wt{\Sb}_N \rangle_{L^2} - \langle D_N^{1/2} \wt{\Sb}_N, D_N^{1/2} \bm{\phi} \times \wt{\Sb}_N \rangle_{L^2} \\
& \quad + \langle D_N^{1/2} \wt{\Sb}_N, D_N^{1/2} (\bm{\phi} \times \wt{\Sb}_N) \rangle_{L^2} \\
& = \underbrace{\langle D^{1/2}_N \bm{\phi}, \wt{\Sb}_N \times D_N^{1/2} \wt{\Sb}_N \rangle_{L^2}}_{=: {\tt I}}- \underbrace{\langle D^{1/2}_N \wt{\Sb}_N, [D_N^{1/2}, \wt{\Sb}_N \times] \bm{\phi} \rangle_{L^2}}_{=: {\tt II}}.
\end{align*}
We bound the terms {\tt I} and {\tt II} as follows. Recall that $\wh{\Sb}_k$ denotes the $k$-th Fourier coefficients of $\wt{\Sb}_N$. Similarly we write $\wh{\bm{\phi}}_k$ for the Fourier coefficients of $\bm{\phi}$. We find
\begin{align*}
| {\tt I} | & = |\langle D^{1/2}_N \bm{\phi}, \wt{\Sb}_N \times D_N^{1/2} \wt{\Sb}_N \rangle_{L^2} | =  |\langle D^{1/2}_N \wt{\Sb}_N, D_N^{1/2} \bm{\phi} \times \wt{\Sb}_N \rangle_{L^2} | \\
& \leq \| D^{1/2}_N \wt{\Sb}_N \|_{L^2} \| D_N^{1/2} \bm{\phi} \times \wt{\Sb}_N \|_{L^2} \lesssim \| \wt{\Sb}_N \|_{H^{1/2}} \| D_N^{1/2} \bm{\phi} \times \wt{\Sb}_N \|_{L^2}.
\end{align*}
Furthermore, we observe (where we denote $\mu_j = |j|/2$ for $|j| > n$ in the following) that
\begin{align} \label{eq:cont_div_1}
     \norm{D_N^{1/2} \bm{\phi} \times \wt{\Sb}_N}_{L^2}^2 &= \sum_{k}\Big|\sum_{j}\mu_j^{1/2}\wh{\bm{\phi}}_j\wedge \wh{\Sb}_{k-j} \Big|^2 \nonumber \\
                    &\le \sum_{k}\bigg(\sum_{j}\frac{1}{\langle k-j\rangle \langle j \rangle^{2\eps}}\bigg)\sum_j \langle k-j\rangle|\wh{\Sb}_{k-j}|^2 \langle j \rangle^{2\eps} \mu_j|\wh{\bm{\phi}}_j|^2\\
                    & \lesssim \sum_{k=1}^\infty \frac{1}{\langle k \rangle^{1+2\eps}} \norm{\wt{\Sb}_N}_{H^{1/2}}^2 \norm{\langle \nabla \rangle^{1+2\eps} \bm{\phi}}_{L^2}^2 \lesssim \norm{|\nabla|^{1/2} \wt{\Sb}_N}_{L^2}^2 \norm{\langle \nabla \rangle^{1/2+\eps} \bm{\phi}}_{L^2}^2. \nonumber
\end{align}
Thus we have found that
\be
| {\tt I} |  \lesssim \| |\nabla|^{1/2} \wt{\Sb}_N \|_{L^2}^4 \| \bm{\phi} \|_{H^{1/2+\eps}}^2 \lesssim \| \bm{\phi} \|_{H^{1/2+\eps}}^2
\ee
thanks to the a-priori bound $\| |\nabla|^{1/2} \wt{\Sb}_N \|_{L^2} \lesssim \| D_N^{1/2} \Sb_N \|_{\ell^2(\Ss_N)} \lesssim 1$ uniformly in $N$. 

Next, we turn to the term {\tt II}. Due to the fact that $D_N^{1/2} \wt{\Sb}_N \in \mathrm{Pol}_n$ and $\mathbf{a} \cdot (\mathbf{a} \times \mathbf{b})=0$ for any $\mathbf{a}, \mathbf{b} \in \R^3$, we can write
\be
{\tt II}=\cip{D_N^{1/2} \wt{\Sb}_N}{P_nQ_N(\wt{\Sb}_N,\bm{\phi})}_{L^2},
\ee
where we set
\be
   Q_N(\wt{\Sb}_N,\bm{\phi}):= D_N^{1/2}(\wt{\Sb}_N\times \bm{\phi})-(D_N^{1/2}\wt{\Sb}_N)\times \bm{\phi} -\wt{\Sb}_N\times D_N^{1/2} \bm{\phi}.
\ee
 To estimate this term, we use the uniform bounds
\be\label{eq:bound_difference}
|\mu_\ell^{1/2} - \mu_j^{1/2}| \lesssim |\ell-j|^{1/2} \quad \mbox{and} \quad \mu_{j-k} \lesssim |j-k|
\ee 
for all $\ell, j \in \Z$, which follows from straightforward arguments (for $0\le j\le n<\ell$, we remark that $0<\mu_\ell^{1/2}-\mu_{j}^{1/2}\le 2^{-1/2}(\ell^{1/2}-j^{1/2})$). 
Using these estimates, we find
\begin{align}\label{eq:cont_div_2}
     \norm{P_nQ_N(\wt{\Sb}_N,\bm{\phi})}_{L^2}^2 &\leq \sum_{j=-n}^n\bigg[\sum_{k=-n}^{n} |(\mu_{j}^{1/2}-\mu_k^{1/2})-\mu_{j-k}^{1/2}||\wh{\bm{\phi}}_{j-k}||\wh{\Sb}_k|\bigg]^2 \\
     & \lesssim \sum_{j=-n}^n\bigg(\sum_{k=-n}^n\frac{1}{\langle j-k \rangle^{2\eps} \langle k\rangle}\bigg)\sum_{k=-n}^n \langle j-k \rangle^{2\eps} |j-k| |\wh{\bm{\phi}}_{j-k}|^2 \langle k\rangle |\wh{\Sb}_k|^2  \nonumber \\
     & \lesssim \sum_{k=1}^\infty \frac{1}{\langle k \rangle^{1+2\eps}} \norm{\wt{\Sb}_N}_{H^{1/2}}^2 \norm{\langle \nabla \rangle^{1/2+\eps} \bm{\phi}}_{L^2}^2 \lesssim \norm{\wt{\Sb}_N}_{H^{1/2}}^2  \norm{\langle \nabla \rangle^{1/2+\eps} \bm{\phi}}_{L^2}^2 \nonumber
 \end{align}
 Again, by using the a-priori bound $\| \wt{\Sb}_N \|_{H^{1/2}} \lesssim \| \langle D_N \rangle^{1/2} \Sb_N \|_{\ell^2(\Ss_N)} \lesssim 1$ for all $N$, we deduce that
 $$
 |{\tt II}| = |\cip{D_N^{1/2} \wt{\Sb}_N}{P_nQ_N(\wt{\Sb}_N,\bm{\phi})}_{L^2}| \lesssim \|\wt{\Sb}_N \|_{H^{1/2}}^2 \| \bm{\phi} \|_{H^{1/2+\eps}} \lesssim \| \bm{\phi} \|_{H^{1/2+\eps}} .
 $$
 This completes the proof of Lemma \ref{lem:apriori}.
 \end{proof}

\subsection{Proof of Theorem \ref{thm:main_periodic}}

Let $T >0$ and $\eps > 0$ be given. Suppose that $\{\Sb_{N,0} : \Ss_N \to \Ss^2 \}_{N \in 2\N+1}$ is a family of initial data for the evolution equation \eqref{eq:Spin_per_class} with $\| \wt{\Sb}_{N,0} \|_{H^{1/2}(\Ss)} \lesssim 1$ independent of $N$. Let $\wt{\Sb}_N : [0,\infty) \times \Ss \to \R^3$ be the corresponding trigonometric interpolation as given by Proposition \ref{prop:trig_flow}. We divide the rest of the proof into the following steps. 

\medskip
{\bf Step 1.} By conservation laws and Lemma \ref{lem:apriori}, we have the a-priori bounds
\be
\sup_{t \in [0,T]} \| \wt{\Sb}_N(t) \|_{H^{1/2}} \lesssim 1 \quad \mbox{and} \quad \sup_{t \in [0,T]} \|\pt_t \wt{\Sb}_N(t) \|_{H^{-1/2-\eps}} \lesssim 1
\ee
independent of $N$. Hence, after passing to a subsequence if necessary, we can assume that
\be\label{eq:weak_limit_S_N}
\mbox{$\wt{\Sb}_N \weakto \Sb$ weakly-$*$ in $L^\infty\big([0,T]; H^{1/2}(\Ss)\big) \cap W^{1,\infty}\big([0,T]; H^{-1/2-\eps}(\Ss)\big)$ as $N \to \infty$.}
\ee
As a consequence, for any fixed integer $\ell \in \N$, it holds that 
\be \label{eq:conv_strong_Pell}
\mbox{$P_\ell \wt{\Sb}_N(t) \to P_\ell \Sb(t)$ strongly in $L^\infty\big([0,T]; L^2(\Ss)\big)$ as $N \to \infty$,}
\ee
where we recall that $P_\ell : L^2(\Ss) \to \mathrm{Pol}_\ell$ denotes the projection onto the finite-dimensional subspace of trigonometric polynomials of degree at most $\ell$.

Next, we claim that $\Sb \in L^\infty\big([0,T]; H^{1/2}(\Ss)\big) \cap W^{1,\infty}\big([0,T]; H^{-1/2-\eps}(\Ss)\big)$ satisfies 
\be \label{eq:S_limit}
\langle \bm{\phi}, \pt_t \Sb(t) \rangle = \langle |\nabla|^{1/2} \Sb(t), |\nabla|^{1/2} (\bm{\phi} \times \Sb(t)) \rangle \quad \mbox{for a.\,e.~$t \in [0,T]$}
\ee
for any $\bm{\phi} \in H^{1/2+\eps}(\Ss; \R^3)$. 

To prove \eqref{eq:S_limit}, we first recall that $\pt_t \wt{\Sb}_N \weakto \pt_t \Sb$ weakly-$*$ in $L^\infty([0,T]; H^{-1/2-\eps})$. Let $\mathbf{U} \in L^1([0,T]; H^{1/2+\eps}(\Ss;\R^3))$ be given. Thus we have
\be \label{eq:conv_Sb_N_1}
\int_0^T \langle \mathbf{U}(t), \pt_t \wt{\Sb}_N(t) \rangle \, dt \to \int_0^T \langle \mathbf{U}(t), \pt_t \Sb(t) \rangle \, dt \quad \mbox{as} \quad N \to \infty.
\ee
Next, we recall the definition of the error term $\EE(\wt{\Sb}_N(t))$ from \eqref{def:E_error} above. By applying Lemma \ref{lem:error}, we deduce
\begin{align*}
& \int_0^T \langle \mathbf{U}(t), \cI_N (\wt{\Sb}_N(t) \times D_N \wt{\Sb}_N(t) ) \rangle \, dt \\
& = \int_0^T \langle \mathbf{U}(t), \wt{\Sb}_N(t) \times D_N \wt{\Sb}_N(t) \rangle \, dt + \underbrace{\int_0^T \langle \mathbf{U}(t), \EE(\wt{\Sb}_N(t)) \rangle \, dt}_{=o(1)} \\
& = \int_0^T \left ( \langle D_N^{1/2} \mathbf{\Sb}_N(t), D_N^{1/2} \mathbf{U}(t) \times \wt{\Sb}_N(t) \rangle + \langle D_N^{1/2} \wt{\Sb}_N(t), P_n Q_N(\wt{\Sb}_N, \mathbf{U}) \rangle \right ) dt + o(1)
\end{align*}
with $o(1) \to 0$ as $N \to \infty$ and where we adopt the notation used in the proof of Lemma \ref{lem:apriori} above. We now claim that
\be \label{eq:conv1}
\mbox{$D_N^{1/2} \mathbf{U} \times \wt{\Sb}_N \to |\nabla|^{1/2} \mathbf{U} \times \Sb$ as $N \to \infty$},
\ee
\be \label{eq:conv2}
\mbox{$P_n Q_N(\wt{\Sb}_N, \mathbf{U}) \to  \EE_{\mathrm{KPV}}(\Sb, \mathbf{U})$ as $N \to \infty$}
\ee
with strong convergence in $L^1([0,T]; L^2(\Ss))$, where $\EE_{\mathrm{KPV}}(\Sb, \mathbf{U})$ denotes the \textbf{Kato--Ponce--Vega commutator-type term} defined as
\be \label{def:KPV}
\EE_{\mathrm{KPV}}(\Sb, \mathbf{U}) = |\nabla|^{1/2}( \Sb \times \mathbf{U}) - |\nabla|^{1/2} \Sb \times \mathbf{U} - \Sb \times |\nabla|^{1/2} \mathbf{U}.
\ee

 To see that \eqref{eq:conv1} holds true, we let $\ell \in \N$ be an integer number and we consider the $\ell$-splitting %, $\ell\in\mathbb{N}$
 \[
 	D_N^{1/2} \mathbf{U} \times \wt{\Sb}_N=D_N^{1/2} \mathbf{U} \times P_\ell\wt{\Sb}_N+D_N^{1/2} \mathbf{U} \times P_\ell^\perp\wt{\Sb}_N.
 \]
 For a fixed $\ell>0$, we have $D_N^{1/2} \mathbf{U} \times P_\ell\wt{\Sb}_N\to |\nabla|U\wedge P_\ell \wt{\Sb}_N$ strongly in $L^1([0,T];L^2)$. Proceeding as in \eqref{eq:cont_div_1}, we get the point-wise estimate
 \[
 	\norm{D_N^{1/2} \mathbf{U} \times P_\ell^\perp\wt{\Sb}_N}_{L^2}\lesssim \norm{U}_{H^{1/2+\eps}}\norm{\wt{\Sb}_N}_{H^{1/2}}\times \underset{\ell\to+\infty}{o}(1).
 \]
 We have used the rearrangement inequality to get for all $k\in\Z$
 \[
 \sum_{j\in\Z\atop \ell<|k-j|\le n}\frac{1}{\langle k-j\rangle\langle j\rangle^{2\eps}}\le \sum_{m=-n}^n\frac{1}{\langle \ell+|m|\rangle\langle m\rangle^{2\eps}}=\underset{\ell\to+\infty}{o}(1).
 \]
 We conclude that \eqref{eq:conv1} holds true by a straightforward $\eps$-argument.
 
 %we notice that, by Rellich compactness of $H^{1/2}(\Ss) \subset L^2(\Ss)$, it follows that $\wt{\Sb}_N(t) \to \Sb(t)$ strongly in $L^2(\Ss)$ for almost every $t \in [0,T]$. By dominated convergence, we readily deduce that \eqref{eq:conv1} holds.

Similarly, for the proof of \eqref{eq:conv2} we let $\ell \in \N$ and we write
\begin{align*}
P_nQ_N (\wt{\Sb}_N, \mathbf{U}) - \EE_{\mathrm{KPV}}(\Sb, \mathbf{U})  & = P_n Q_N(P_\ell \wt{\Sb}_N, \mathbf{U}) - \EE_{\mathrm{KPV}}(P_\ell \Sb, \mathbf{U})  \\
& \quad + P_n Q_N (P_\ell^\perp \wt{\Sb}_N, \mathbf{U}) -  \EE_{\mathrm{KPV}}(P_\ell^\perp \Sb, \mathbf{U}).
\end{align*}
Using \eqref{eq:conv_strong_Pell} it is straightforward to check that
$$
\| P_n Q_N(P_\ell \wt{\Sb}_N, \mathbf{U}) - \EE_{\mathrm{KPV}}(P_\ell \Sb, \mathbf{U}) \|_{L^1([0,T]; L^2)} \to  0 \quad \mbox{as} \quad N \to \infty
$$
for every $\ell \in \N$. Furthermore, by arguing like in the proof of estimate \eqref{eq:cont_div_2} above, we deduce that
\begin{align*}
& \| P_n Q_N(P_\ell^\perp \wt{\Sb}_N, \mathbf{U}) \|_{L^1([0,T]; L^2)} \\
& \lesssim \| \mathbf{U} \|_{L^1([0,T]; H^{1/2+\eps})} \| \wt{\Sb}_N \|_{L^\infty([0,T]; H^{1/2})} \left ( \sum_{k=1}^\infty \frac{1}{\langle k \rangle^{2 \eps} \langle k + \ell \rangle} \right ) \to 0 \quad \mbox{as $\ell \to \infty$}
\end{align*}
uniformly in $N$. Finally, we readily verify that $\| \EE_{\mathrm{KPV}}(P_\ell^\perp \Sb, \mathbf{U}) \|_{L^1([0;T]; L^2)} \to 0$ as $\ell \to \infty$. Therefore we conclude that \eqref{eq:conv2} holds true.

Since $D_N^{1/2} \wt{\Sb}_N \weakto |\nabla|^{1/2} \Sb$ weakly-$*$ in $L^\infty([0,T]; L^2)$, a summary of our findings shows
\begin{align*}
& \lim_{N \to \infty} \int_0^T \langle \mathbf{U}(t), \cI_N (\wt{\Sb}_N(t) \times D_N \wt{\Sb}_N(t) ) \rangle \, dt \\
& = \int_0^T \left ( \langle |\nabla|^{1/2} \Sb(t), |\nabla|^{1/2} \mathbf{U}(t) \times \Sb(t) \rangle + \langle |\nabla|^{1/2} \Sb(t), \EE_{\mathrm{KPV}}\big(\Sb(t), \mathbf{U}(t)\big) \rangle \right ) dt \\
& = \int_0^T \langle |\nabla|^{1/2} \Sb(t), |\nabla|^{1/2}\big(\mathbf{U}(t) \times \Sb(t)\big) \rangle \, dt,
\end{align*}
taking into account the elementary facts that $|\nabla|^{1/2} \Sb \cdot (|\nabla|^{1/2} \Sb \times \mathbf{U})=0$ and $|\nabla|^{1/2} \Sb \times \mathbf{U} = -\mathbf{U} \times |\nabla|^{1/2} \Sb$. In view of \eqref{eq:conv_Sb_N_1} together with $\pt_t \wt{\Sb}_N = \cI_N (\wt{\Sb}_N \times D_N \wt{\Sb}_N)$, we conclude that
\be
\int_0^T \langle \bm{U}(t), \pt_t \Sb(t) \rangle \, dt = \int_0^T \langle  |\nabla|^{1/2} \Sb(t), |\nabla|^{1/2}\big(\mathbf{U}(t) \times \Sb(t)\big) \rangle \, dt
\ee
for all $\bm{U} \in L^1([0,T]; H^{1/2+\eps}(\Ss; \R^3))$. As a consequence, we obtain
\be \label{eq:conv_S_pre}
\langle \bm{\phi}, \pt_t \Sb(t) \rangle =  \langle  |\nabla|^{1/2} \Sb(t), |\nabla|^{1/2}\big(\bm{\phi} \times \Sb(t)\big) \rangle \quad \mbox{for a.\,e.~$t \in [0,T]$}
\ee
for all $\bm{\phi}\in H^{1/2+\eps}(\Ss; \R^3)$. In the next step, we will upgrade this to all $\bm{\phi} \in H^{1/2}(\Ss; \R^3)$ by showing that $|\Sb(t,x)|^2 \le 1$ for a.\,e.~$(t,x) \in [0,T] \times \Ss$, which will imply that the right-hand side define a bounded map for $\bm{\phi} \in H^{1/2}$.

\medskip
{\bf Step 2.} Although we cannot directly control the $L^\infty$-norm of the trigonometric interpolations $\wt{\Sb}_N(t)$, we can show that the limit $\Sb$ is indeed $\Ss^2$-valued as follows. Let $\bm{\Sigma}_N:[0,T] \times \Ss \to \R^3$ denote the piecewise constant interpolation of $\Sb_N(t)$ as given by \eqref{eq:piece_wise_cst_extension}. By construction and the fact that $|\Sb_{N}(t,x)|^2 = 1$ for all $(t,x) \in [0,T] \times \Ss_N$, we have that $|\bm{\Sigma}_N(t,x)|^2 = 1$ for a.\,e.~$t$ and $x$. Moreover, by Lemma \ref{lem:compare_piece_cst_T}, we also note (after passing to a subsequence if necessary) that
\be \label{eq:conv_L2_strong}
\| \wt{\Sb}_N(t) - \bm{\Sigma}_N(t) \|_{L^\infty([0,T]; L^2(\Ss))} \to 0 \quad \mbox{as} \quad N \to \infty.
\ee
From this we readily deduce that the limit satisfies $|\Sb(t,x)|^2 \le 1$ for a.\,e.~$t$ and $x$: up to an extraction $\boldsymbol{\Sigma}_N\rightharpoonup \Sb$ $*$-weakly in $L^\infty([0,T]\times\Ss)$ so that $|\Sb(t,x)|\le 1$ for a.e. $(t,x)$. Equality
$$
|\Sb(t,x)| = 1 \quad \mbox{for a.\,e.~$(t,x) \in [0,T] \times \Ss$} 
$$
is ensured once proven that the weak limit is  weak solutions to (HWM) (by conservation of the $L^2$-norm for weak solutions). Let us check it.

We claim that the right-hand side in \eqref{eq:conv_S_pre} is a bounded map in $\bm{\phi} \in H^{1/2}$ uniformly in $t \in [0,T]$. Indeed, recalling \eqref{def:KPV} and using once again that $|\nabla|^{1/2} \Sb \cdot (|\nabla|^{1/2} \Sb \times \mathbf{U})=0$, we find
\begin{align*}
& \left |  \langle  |\nabla|^{1/2} \Sb(t), |\nabla|^{1/2}(\bm{\phi} \times \Sb(t)) \rangle \right |  =  \left | \EE_{\mathrm{KPV}}(\Sb(t), \bm{\phi}) + \langle |\nabla|^{1/2} \Sb(t), \Sb(t) \times |\nabla|^{1/2} \bm{\phi} \rangle \right | \\
& \lesssim \|  |\nabla|^{1/4} \Sb(t) \|_{L^4} \| |\nabla|^{1/4} \bm{\phi} \|_{L^4} + \| |\nabla|^{1/2} \Sb(t) \|_{L^2} \| \Sb(t) \|_{L^\infty} \| |\nabla|^{1/2} \bm{\phi} \|_{L^2} \\
& \lesssim  \|  |\nabla|^{1/2} \Sb(t) \|_{L^2} \| |\nabla|^{1/2} \bm{\phi} \|_{L^2} + \| \nabla^{1/2} \Sb(t) \|_{L^2} \| \|\nabla|^{1/2} \bm{\phi} \|_{L^2} \lesssim \| |\nabla|^{1/2} \bm{\phi} \|_{L^2}
\end{align*}
using the Cauchy-Schwarz inequality, the classical Kato--Ponce estimate and Sobolev embeddings together with a-priori bounds $\| |\nabla|^{1/2} \Sb(t)\|_{L^2} \lesssim 1$ and $\|\Sb(t) \|_{L^\infty}\le 1$. Hence we conclude that \eqref{eq:conv_S_pre} extends to all $\bm{\phi} \in H^{1/2}(\Ss; \R^3)$ and $\pt_t \Sb \in L^\infty([0,T]; H^{-1/2}(\Ss))$. This shows that the limit $\Sb$ is a weak solution of (HWM) satisfying $\Sb(t,x) \in \Ss^2$ for a.\,e.~$(t,x)$.

Finally, we note that we have the strong convergence
\be
\| \wt{\Sb}_N - \Sb\|_{L^2([0,T]; H^{1/2-\eps}(\Ss))} \to 0 \quad \mbox{as} \quad N \to \infty
\ee
for any $\eps > 0$. Indeed, this follows from simple interpolation of the proven strong convergence $\| \wt{\Sb}_N - \Sb \|_{L^2([0,T]; L^2(\Ss))} \to 0$ as $N \to \infty$ together with the a-priori bound $\| \wt{\Sb}_N - \Sb \|_{L^\infty([0;T]; H^{1/2}(\Ss))} \lesssim 1$ independent of $N$. 

This completes the proof of Theorem \ref{thm:main_periodic}. \hfill $\qed$

\section{Convergence to Higher Regularity Solutions: Proof of Theorem \ref{thm:conv_strg}}

This section is devoted to the proof of Theorem \ref{thm:conv_strg}. We first show convergence to the higher regularity solutions of (HWM) on short time intervals by adapting the vanishing viscosity method to the discrete setting. Second, we prove the convergence estimate in Theorem \ref{thm:conv_strg} by a Gr\"onwall argument.

\subsection{Vanishing Viscosity Method in the Discrete Setting} \label{subsec:vanish_visc}

We adapt the method used in \cite{PuGu-13} to the present discrete setting: by Kato's vanishing viscosity method we establish local-in-time uniqueness of solutions to (HWM) for regular initial data.

For $\eps > 0$, we consider the following `parabolic' globally well-posed regularization of the discrete system \eqref{eq:Spin_per_class} given by
\be\label{eq:disc_visc_van_circle}\tag{Eq$_{N,\eps}$}
\left \{ \begin{array}{l} \displaystyle
\partial_t \Sb_N^{\eps}=\Sb_N^{\eps}\times |\nabla|_N\Sb_N^{\eps}+\eps\Delta_N \Sb_N^{\eps}, \\
\Sb_N^\eps |_{t=0} = \Sb_{N,0},
\end{array} \right .
\ee
with some initial condition $\Sb_{N,0} : \Ss_N \to \Ss^2$ obtained from sampling of some given initial data $\Sb_0 \in H^{5/2}(\Ss; \Ss^2)$ for  the continuum equation (HWM).  We recall that $\Delta_N$ denotes the discrete Laplacian on $\Ss_N = \{ z_k = e^{2\pi k \im/N} : k=0,\ldots,N-1\}$ with
\[
\forall\, k\in \mathbb{Z},\,
\left\{
	\begin{array}{ccl}
 		-[\Delta_N \mathbf{V}]_k(z_k)&=&\tfrac{(2\pi)^2}{N^2}(\mathbf{V}(z_{k+1})+\mathbf{V}(z_{k-1})-2\mathbf{V}(z_k)),\\
		\pm [D_{\pm ,N}\mathbf{V}]_k&=&\pm \tfrac{2\pi}{N}(\mathbf{V}(z_{k\pm 1})-\mathbf{V}(z_k)).
	\end{array}
\right.
\]
It is convenient to introduce the parameter
$$
h_N:=\frac{2\pi}{N}.
$$ 
Straightforward computations show that $-\Delta_N=D_{+,N}D_{-,N}$ is conjugated to the Fourier multiplier (acting on $\mathrm{Pol}_n$)
\be\label{eq:corr_Delta}
	-M_{\Delta_N}(k)=\mathrm{sinc}\Big(\frac{h_Nk}{2}\Big)^2k^2\ \Big(\ge \frac{4}{\pi^2}k^2\ \text{for }-n\le k\le n\Big),
\ee
and that we have
\be\label{eq:corr_D_eps}
	M_{D_{+,N}}(k)=\overline{M_{D_{-,N}}(k)}=ik e^{ikh_N/2}\mathrm{sinc}\big( \tfrac{kh_N}{2}\big). 
\ee
Note that by direct computations $\norm{\Sb_N^\eps}_{\ell^2}^2$ and $\norm{\Sb_N^\eps}_{H^{1/2}(\Ss_N)}^2$ defined by
\[
\norm{\Sb_N^\eps}_{H^{1/2}(\Ss_N)}^2:=\norm{\Sb_N^\eps}_{\ell^2}^2+\cip{|\nabla|_N\Sb_N^\eps}{\Sb_N^\eps}_{\ell^2}
\]
are seen to be Lyapounov functions for \eqref{eq:disc_visc_van_circle}, i.\,e., they are non-increasing. Furthermore, we introduce the discrete higher Sobolev norm for maps $\mathbf{V}$ on $\Ss_N$ defined as
 \be\label{eq:def_discrete_Sob}
	\norm{\mathbf{V}}_{H^{5/2}(\Ss_N)}^2:= \norm{\mathbf{V}}_{\ell^2}^2+\norm{D_{+,N}^2 |\nabla|_N^{1/2}\mathbf{V}}_{\ell^2}^2.
\ee 
Our aim is to prove the Gr\"{o}nwall estimate
\be\label{eq:gronwall_discrete}
	\boxed{ \partial_t\norm{\Sb_N^\eps}_{H^{5/2}(\Ss_N)}\le \Gamma\norm{\Sb_N^\eps}_{H^{5/2}(\Ss_N)}^2 }
\ee
with some suitable constant $\Gamma>0$ depending only on the initial data $\Sb_0$ (and thus independent of $\eps$ and $N$). From \eqref{eq:gronwall_discrete} we can deduce 
\be\label{eq:gronwall_on_H_sigma}
	\norm{\Sb_N^{\eps}(t)}_{H^{5/2}(\Ss_N)}\le \frac{\norm{\Sb_{N,0}}_{H^{5/2}(\Ss_N)}}{1-\Gamma t\norm{\Sb_{N, 0}}_{H^{5/2}(\Ss_N)}}, \quad 0\le t<\big[ \Gamma \norm{\Sb_{N,0}}_{H^{5/2}(\Ss_N)}\big]^{-1}.
\ee

Assume now that \eqref{eq:gronwall_discrete} holds true for the moment. Let us explain how this implies Theorem~\ref{thm:conv_strg} as follows.
Taking the limit $\eps\to 0^+$, this shows that the solution $\Sb_N(t)$ also satisfies \eqref{eq:gronwall_discrete} and \eqref{eq:gronwall_on_H_sigma}. %\[
%	\norm{\Sb_N^\eps(t)}_{H^{\sigma}(\Ss_N)}\le f(\Gamma,t),\ \lim_{t\to T_*}f(\Gamma,t)=+\infty.
%\]
In view of Proposition \eqref{prop:trig_inter}, we get a similar $H^{5/2}$-bound on its trigonometric interpolation $\wt{\Sb}_N(t)$, up to this same time $T_*$ and in a $N$-independent fashion. Taking the weak limit along any weakly converging subsequence $N\to+\infty$ as given by Theorem \ref{thm:main_periodic} we obtain the same $H^{5/2}$-bound for the weak limit. The latter has initial data $\Sb_{(0)}$, we deduce from the uniqueness result in Lemma~\ref{lem:unique_reg} that it coincides up to time $T_*(\Gamma,\norm{\Sb_0}_{H^{5/2}})>0$ with the unique regular strong solution $\Sb \in C^0([0,T_*); H^{5/2})$ of (HWM).

\begin{remark*}
For $-n\le k\le n$, we have the equivalence
$$
-M_{\Delta_N}(k) = \mathrm{sinc} \Big(\frac{h_Nk}{2}\Big)^2k^2 \sim k^2 \left (1- \frac{|k|}{N} \right )^2 = \widehat{D_N^2(k)}
$$
uniformly in $N=2n+1$. When estimating norms, we could interchange $-\Delta_N$ with $|\nabla_N|^2$ and vice versa. We have made the choice \eqref{eq:def_discrete_Sob} to simplify the computations below.
\end{remark*}

\subsection{Proof of the Gr\"{o}nwall Estimate \eqref{eq:gronwall_discrete}} \label{subsec:groen}

Inspired by the discussion \cite{PuGu-13} that deals with the parabolic regularization of (HWM), we consider the quantity
\be
Q_{N}^\eps(t) := \tfrac{1}{2}\partial_t\big[\cip{D_{+,N}^2 |\nabla|_N \Sb_N^{\eps}}{D_{+,N}^2 \Sb_N^{\eps}} \big]+\eps\cip{D_{+,N}^{3} |\nabla|_N \Sb_N^{\eps}}{D_{+,N}^{3}\Sb_N^{\eps}}
\ee
for the solution $\Sb_N^{\eps}$ of \eqref{eq:disc_visc_van_circle}. An elementary calculation yields
\be\label{eq:disc_local_strg_der_to_est}
	Q_N^\eps =\cip{D_{+,N}^2(|\nabla|_N \Sb_N^{\eps})}{D_{+,N}^2([|\nabla|_N \Sb_N^\eps]\times \Sb_N^\eps)}.
\ee
We use the version of the discrete analogue of Leibniz' product formula for lattice functions  $f,g$ on $\Ss_N$:
\be\label{eq:TF_D_+}
	D_{+,N}(fg)=D_{+,N}(f)T_N(g)+fD_{+,N}(g),
\ee
where
 \be
	 (T_Ng)_\theta=g_{\theta+h_N},\ \theta\in h_N\Z/(2\pi \Z),\ h_N=\tfrac{2\pi}{N},
\ee
is the translation operator $T_N$ which commutes with $D_{+,N}$, and the arguments $\theta$ in $(T_Ng)_\theta$ are the angle of the $N$-th root of unity.
Iterating the formula \eqref{eq:TF_D_+} yields the equality $D_{+,N}^m(fg)=\sum_{k=0}^{m} \binom{m}{k}D_{+,N}^k(f)D_{+,N}^{m-k}(T_N^{k}g)$ for any $m \in \N$. By taking $f=|\nabla|_N \Sb_N^\eps$ and $g=\Sb_N^\eps$ with $m=2$ we get:
\begin{multline}\label{eq:formule_vanish_meth}
Q_N^\eps = \cip{D_{+,N}^2(|\nabla|_N \Sb_N^{\eps})}{D_{+,N}^2([|\nabla|_N \Sb_N^\eps]\times  \Sb_N^\eps)}  \\
		=\sum_{k=0}^{1}\binom{2}{k}\cip{[D_{+,N}^2|\nabla|_N \Sb_N^{\eps}]}{[D_{+,N}^{k}|\nabla|_N \Sb_N^\eps]\times [ D_{+,N}^{2-k}(T_N^k\Sb_N^\eps)]},\\
		=\sum_{k=0}^{1}\binom{2}{k}\cip{|\nabla|_N ^{1/2}[D_{+,N}^2 \Sb_N^{\eps}]}{|\nabla|_N ^{1/2}([D_{+,N}^{k}|\nabla|_N \Sb_N^\eps]\times [ D_{+,N}^{2-k}(T_N^k \Sb_N^\eps)])}.
\end{multline}
Note we have the sum $\sum_{k=0}^{1}$ only, since the top-order term with $k=2$ vanishes. To conclude, we need to estimate the $\ell^2$-norm of each of the terms involved in the sum.

By the isometry $\ell^2(\Ss_N)\,\widehat{=}\,\text{Pol}_n$ with $N=2n+1$, we can interpret the inner products in \eqref{eq:formule_vanish_meth} as $L^2(\Ss)$-inner products of the trigonometric polynomials given by corresponding trigonometric interpolations. For $k=0,1$, we set
\be \label{def:F_G}
F_k=D_{+,N}^{k}|\nabla|_N \Sb_N^\eps \quad \mbox{and} \quad  G_k= D_{+,N}^{2-k}(T_N^k \Sb_N^\eps).
\ee 
We first claim that
\be\label{eq:inter_to_show}
\norm{D_N ^{1/2}(\wt{F_k\wedge G_k})} = \norm{ |\nabla|^{1/2}\cI_N(\wt{F}_k\wedge \wt{G}_k)} \lesssim \norm{|\nabla|^{1/2}(\wt{F}_k\wedge \wt{G}_k)}.
\ee
Assuming this estimate holds, we can use the fractional Leibniz rule (Kato-Ponce-Vega), Sobolev inequalities  and Proposition \ref{prop:trig_inter} to deduce that
\be\label{eq:inter_to_show_2}
\norm{ |\nabla|^{1/2}(\wt{F}_k\wedge \wt{G}_k)}_{L^2}\lesssim \norm{\Sb_N^\eps}_{\ell^2}^2+\norm{D_{+,N}^2|\nabla|_N^{1/2} \Sb_N^\eps}_{\ell^2}^2 = \norm{\Sb_N^\eps}_{H^{5/2}(\Ss_N)}^2 
\ee 
From this we readily conclude that  \eqref{eq:gronwall_discrete} holds by using the Cauchy-Schwarz inequality in \eqref{eq:formule_vanish_meth}. 

Thus is remains to prove \eqref{eq:inter_to_show} and \eqref{eq:inter_to_show_2}, which will be done in the next subsection.

\subsection{Functional Inequalities in the Discrete Setting}\label{eq:transposition_fun_ineq}

We are now going to prove the claimed estimates \eqref{eq:inter_to_show} and \eqref{eq:inter_to_show_2} as follows. 

As usual, we suppose that $N=2n+1$ in what follows. By using Proposition \ref{prop:trig_inter} (iii), we deduce that for all functions $f,g$ defined on $\Ss_N$ and $s \geq 0$ that
\be\label{eq:H_s_product_formula}
	\begin{array}{rcl}
		\norm{\Dh^s\widetilde{(f\odot g)}}_{L^2}&=&\norm{\Dh^s[P_n+\overline{z}^NP_{n,+}^\perp+z^NP_{n,-}^{\perp}](\wt{f} \odot\wt{g})}_{L^2},\\
			&\le&\norm{\Dh^sP_n(\wt{f}\odot\wt{g})}_{L^2}+\norm{\Dh^s(\overline{z}^NP_{n,+}^\perp+z^NP_{n,-}^{\perp}](\wt{f} \odot\wt{g}))}_{L^2},\\
                &\le&\norm{\Dh^sP_n(\wt{f}\odot\wt{g})}_{L^2}+\norm{[\overline{z}^N(\Dh^sP_{n,+}^\perp)+z^N(D^sP_{n,-}^{\perp})](\wt{f}\odot\wt{g})}_{L^2},\\ 
						&\le&\norm{\Dh^sP_n(\wt{f}\odot\wt{g})}_{L^2}+\norm{\Dh^sP_n^\perp(\wt{f}\odot\wt{g})}_{L^2},\\
						&\le&2^{1/2}\norm{\Dh^s(\wt{f}\odot\wt{g})}_{L^2}.
	\end{array}
\ee
Here the operation $\odot$ either denotes usual multiplication of scalar functions or the scalar product $\cdot$ or the vector product $\times$ if $f,g$ are vector-valued. In particular this implies \eqref{eq:inter_to_show}. 

Next, we prove \eqref{eq:inter_to_show_2}. As a first step, we use the fractional Leibniz rule (Kato-Ponce-Vega estimate) and Sobolev inequalities to conclude
\begin{align*}
\norm{\Dh^{1/2}(\wt{F}\wedge \wt{G})}_{L^2(\Ss)^3} &\lesssim \norm{\Dh^{1/2}\wt{F}}_{L^{p_1}}\norm{\wt{G}}_{L^{q_1}}+\norm{\wt{F}}_{L^{p_2}}\norm{\Dh^{1/2}\wt{G}}_{L^{q_2}} \\
			&\lesssim\norm{\Dh^{1/2+s(p_1)}\wt{F}}_{L^2}\norm{\Dh^{s(q_1)}\wt{G}}_{L^2}+\norm{\Dh^{s(p_2)}\wt{F}}_{L^2}\norm{\Dh^{1/2+s(q_2)}\wt{G}}_{L^2} .
\end{align*}
Here $p_i^{-1}+ q_i^{-1}=1/2$ for $i=1,2$ and $s(p_i), s(q_i) \geq 0$  are suitably chosen depending on $k=0,1$ appearing in the definition of $F$ and $G$ in \eqref{def:F_G}. More precisely, for $k=0$ we choose $q_1\in (2,\infty)$, $s(q_1)=p_1^{-1}=\tfrac{q_1-2}{2q_1}$, $s(p_1)=q_1^{-1}$, $p_2=\infty$, $q_2=2$ with $s(2)=0, s(\infty)=1$. In the remaining case $k=1$ we just interchange the roles of $p_i$ and $q_i$.

Next, we need to close the estimates back to the discrete Sobolev norms involving the lattice functions $\Sb_N^\eps$ only. Here we notice that a calculation yields that 
\[
	\wh{(D_+^m T_N^k f)}_j=i^m j^m \mathrm{sinc}\big[ \tfrac{j\pi}{N}\big]^m e^{i[m+2k]\pi/N}\wh{f}_j,\ -n\le j\le n.
\]
For any $m,k \in \N$, we thus deduce the norm equivalence
\be\label{eq:discrete_to_trigo}
		 \norm{\partial_x^m \wt{f}}_{L^2(\Ss)} \sim_m \norm{\wt{(D_+^m T_N^kf)}}_{L^2(\Ss)}.
\ee
Now, by using Proposition \ref{prop:trig_inter} (i) and \eqref{eq:discrete_to_trigo} with $m=2$, we deduce that
\begin{align*}
& \norm{\Dh^{1/2+s(p_1)}\wt{F}}_{L^2}\norm{\Dh^{s(q_1)}\wt{G}}_{L^2}+\norm{\Dh^{s(p_2)}\wt{F}}_{L^2}\norm{\Dh^{1/2+s(q_2)}\wt{G}}_{L^2} \\
& \lesssim \norm{\,|D_{+,N}|^{k+3/2+s(p_1)}\Sb_N^\eps}_{\ell^2}\,\norm{\,|D_{+,N}|^{m-k+s(q_1)}\Sb_N^\eps}_{\ell^2} \\
			&\quad +\norm{\,|D_{+,N}|^{k+1+s(p_2)}\Sb_N^\eps}_{\ell^2}\,\norm{\,|D_{+,N}|^{m-k+1/2+s(q_2)}\Sb_N^\eps}_{\ell^2},\nonumber\\ %\norm{D^{s(p_2)}\wt{F}}_{L^2}\norm{D^{1/2+s(q_2)}\wt{G}}_{L^2}\\
			&\lesssim \norm{\Sb_N^\eps}_{\ell^2}^2+\norm{D_{+,N}^2|\nabla|_N^{1/2}\Sb_N^\eps}_{\ell^2}^2 = \norm{\Sb_N^\eps}_{H^{5/2}(\Ss_N)}^2.
\end{align*}
This completes the proof of \eqref{eq:inter_to_show_2}.

\subsection{Rate of Convergence} \label{subsec:rate}

Let $\Sb \in C([0,T_*); H^{5/2})$ be the unique solution to (HWM) with initial data $\Sb_0 \in H^{5/2}(\Ss; \Ss^2)$  and $\wt{\Sb}_N(t)$ its approximation as given by Proposition \ref{prop:trig_inter}.  For $t\in [0,T_*)$, we can decompose:
\[
\partial_t \wt{\Sb}_N=\wt{\Sb}_N\wedge D \wt{\Sb}_N+\mathrm{R}_N,
\]
where $D=|\nabla|$ and
\[
\mathrm{R}_N=\big[\cI_N(\wt{\Sb}_N\wedge D\wt{\Sb}_N)-\wt{\Sb}_N\wedge D\wt{\Sb}_N\big]-\tfrac{1}{N}\cI_N(\wt{\Sb}_N\wedge D^2\wt{\Sb}_N).
\]
Let us establish a Gr\"{o}nwall inequality for $\norm{\Sb-\wt{\Sb}_N}_{H^{1/2}}$. From calculations stated in \eqref{eq:diff_H_dot_un_demi} and \eqref{eq:diff_L_2}, we get 
\begin{equation}
	\partial_t \norm{\Sb-\wt{\Sb}_N}_{H^{1/2}}^2\le  f_N(t)\norm{\Sb-\wt{\Sb}_N}_{H^{1/2}}^2-2\cip{\Sb-\wt{\Sb}_N}{\mathrm{R}_N}_{H^{1/2}}.
\end{equation}
with some constants $C > 0$ and the function
$$
f_N(t):=C\left ( \norm{\Sb(t)+\wt{\Sb}_N(t)}_{H^{1/2}\cap L^\infty}+ \norm{\Sb(t)+\wt{\Sb}_N(t)}_{H^{5/2}} \right ).
$$
We now claim that
\be \label{eq:claim}
\norm{\mathrm{R}_N(t)}_{H^{1/2}}=\frac{g(\Sb_N(t))}{N},
\ee
where $g$ is a homogeneous expression of degree $2$ in $\norm{\wt{\Sb}_N}_{H^{5/2}}$ and $\norm{\wt{\Sb}_N}_{H^{1/2}}$.

Let us assume that \eqref{eq:claim} holds true. Then, by Cauchy--Schwarz, we deduce the Gr\"{o}nwall inequality
	\[
		\partial_t \norm{\Sb-\wt{\Sb}_N}_{H^{1/2}}\le \tfrac{f_N(t)}{2}\norm{\Sb-\wt{\Sb}_N}_{H^{1/2}}+\norm{\mathrm{R}_N}_{H^{1/2}},
	\]
	whence
	\begin{multline*}
	\norm{\Sb(t)-\wt{\Sb}_N(t)}_{H^{1/2}}\le \overbrace{\norm{\Sb_0-\wt{\Sb}_{N,0}}_{H^{1/2}}+\int_0^t \norm{\mathrm{R}_N(s)}_{H^{1/2}}\dd s}^{=:\alpha_N(t)}\\+\int_0^t  \tfrac{f_N(s)}{2}\norm{\Sb(s)-\wt{\Sb}_N(s)}_{H^{1/2}} d s.
	\end{multline*}
	We conclude
	\be\label{eq:final_gronwall_for_difference}
		\norm{\Sb(t)-\wt{\Sb}_N(t)}_{H^{1/2}}\le \alpha_N(t)+\frac{1}{2}\int_0^t \alpha_N(s)f_N(s)\mathrm{exp}\Big[\int_s^t \frac{f_N(u)}{2} d  u \Big] d s.
	\ee
	By Lemma~\ref{lem:approx}, we have $\norm{\Sb_0-\wt{\Sb}_{N, 0}}_{H^{1/2}}\le \frac{C}{N^{2}}\norm{\Sb_0}_{H^{5/2}}$. We deduce the estimate 
	$$
	\norm{\Sb(t)-\wt{\Sb}_N(t)}_{H^{1/2}}\le \frac{G(t)}{N}
	$$ thanks to \eqref{eq:claim} with some increasing continuous function $G$ (which growths  faster than exponential; we leave the details to the reader). In particular, we obtain the convergence estimate stated in Theorem \ref{thm:conv_strg}.
	
It remains to prove that \eqref{eq:claim} holds indeed true. To this end, we note that, for any $s>0$, 
\begin{align*}
\norm{\cI_N(\wt{\Sb}_N\wedge D\wt{\Sb}_N)-\wt{\Sb}_N\wedge D\wt{\Sb}_N\big]}_{H^{1/2}}&\le 2\norm{P_n^{\perp}[\wt{\Sb}_N\wedge D\wt{\Sb}_N]}_{H^{1/2}},\\
						&\le \frac{2}{(n+1)^s}\norm{D^s(\wt{\Sb}_N\wedge D\wt{\Sb}_N)}_{H^{1/2}}.
\end{align*}
We recall that $P_n^{\perp}=1-P_n$ denotes the $L^2$-orthogonal projection onto $\mathrm{Pol}_n^{\perp}$. Similarly, we have 
\begin{align*}
	\tfrac{1}{N}\norm{\cI_N(\wt{\Sb}_N\wedge D^2\wt{\Sb}_N)}_{H^{1/2}}&\le \frac{2}{N}\norm{\wt{\Sb}_N\wedge D^2\wt{\Sb}_N}_{H^{1/2}},\\
				&\le \frac{2}{N}\big[\norm{\wt{\Sb}_N\wedge D^2\wt{\Sb}_N}_{L^2}+\norm{D^{1/2}\big(\wt{\Sb}_N\wedge D^2\wt{\Sb}_N\big)}_{L^2}\big].
\end{align*}
Using the fractional Leibniz rule (Kato--Ponce--Vega), we get:
\[
	\norm{\mathrm{R}_N}_{H^{1/2}}\le \frac{1}{N}\big(C\norm{\wt{\Sb}_N}_{H^{5/2}}\norm{\wt{\Sb}_N}_{H^{1/2}}+C(\eps)\norm{\wt{\Sb}_N}_{H^{5/2}}^{\eps/2}\norm{\wt{\Sb}_N}_{H^{1/2}}^{1-\eps/2}\norm{\wt{\Sb}_N}_{H^{5/2}}\big)
\]
with arbitrary $\eps \in (0,2)$. In particular, this shows \eqref{eq:claim}.

\subsection{Proof of Theorem \ref{thm:conv_strg}}

By collecting the results from Subsections \ref{subsec:vanish_visc}--\ref{subsec:rate}, we complete the proof of Theorem \ref{thm:conv_strg}. \hfill $\qed$

\begin{appendix}

\section{Proofs of Some Auxiliary Results}\label{sec:aux}

\subsection{Proof of Lemma~\ref{lem:approx}}
We write $\wh{f}_k$ and $\wh{\cI}_k$ the $k$-th Fourier coefficients of $f$ and $\cI_N(f)$ respectively.
%Let $\wh{f}_k$ denote the $k$-th Fourier coefficient of $f$, and let $\wh{\cI}_k$ denote that of $\cI_N(f)$. 
For $-n\le k\le n$, we have $\wh{\cI}_k=\sum_{j\in\mathbb{Z}}\wh{f}_{k+jN}$. Let us first show (i): we have
\begin{align*}
	\norm{\cI_N(f)}_{H^{1/2+\eps}}^2&=\sum_{k=-n}^n |k|^{1+2\eps}\Big|\sum_{j\in\mathbb{Z}}\wh{f}_{k+jN} \Big|^2\\
		&\le \sum_{k=-n}^n |k|^{1+2\eps}\sum_{j\in\mathbb{Z}}\langle k+jN \rangle^{1+2\eps}|\wh{f}_{k+jN}|^2\sum_{j\in\mathbb{Z}}\frac{1}{\langle k+jN \rangle^{1+2\eps}}.
\end{align*}
Since we have for $-n\le k\le n$ 
\begin{align*}
\sum_{j\in\mathbb{Z}}\frac{|k|^{1+2\eps}}{\langle k+jN \rangle^{1+2\eps}}&\le 1+2\sum_{j=1}^{+\infty}\frac{1}{\Big[|k|^{-1-2\eps}+\big(j\tfrac{N}{|k|}-1\big)^2 \Big]^{\eps+1/2}}\\
									&\le 1+2\sum_{j=1}^{+\infty}\frac{1}{j^{1+2\eps}(2-1)^{1+2\eps}}<+\infty,
\end{align*}
we deduce $\norm{\cI_N(f)}_{H^{\eps+1/2}}\le C(\eps)\norm{f}_{H^{\eps+1/2}}$, which is item (i) in Lemma \ref{lem:approx}.

We now prove item (ii) in Lemma \ref{lem:approx}. Using the triangle inequality we split
\[
	\norm{\cI_N(f)-f}_{H^{s}}\le \norm{P_0[\cI_N(f)-P_nf]}_{H^{s}}+\norm{P_0^{\perp}[\cI_N(f)-P_nf]}_{H^{s}}+\norm{(1-P_n)f]}_{H^{s}}
\]
First we have:
\[
 \norm{P_0[\cI_N(f)-P_nf]}_{H^{s}}^2=\Big|\sum_{j \in \mathbb{Z}^*} \widehat{f}_{jN}^{(0)}\Big|^2\le 2\norm{f}_{H^{1/2+\eps}}^2\sum_{j\ge 1}\langle jN\rangle^{-1-2\eps}\le \frac{C}{\eps N^{1+2\eps}}\norm{f}_{H^{1/2+\eps}}^2.
\]

Then, using \eqref{def:cI_N}, we have
\begin{align}\label{eq:est_trig_int_12}
 \norm{P_0^{\perp}[\cI_N(f)-P_nf]}_{H^{s}}^2&=\sum_{1\le |k|\le n}\langle k\rangle^{2s} \Big|\sum_{j \in \mathbb{Z}^*} \widehat{f}_{k+jN}^{(0)}\Big|^2\\
    &\le \sum_{1\le |k|\le n}\langle k\rangle^{2s} \sum_{j\in\Z^*}\langle k+jN\rangle^{1+2\eps}|\widehat{f}_{k+jN}^{(0)}|^2\frac{N}{N}\sum_{j\in\Z^*}\frac{1}{\langle k+jN\rangle^{1+2\eps}}\nonumber \\
    &\le \sum_{1\le |k|\le n}\frac{2\langle k\rangle^{2s}}{N} \sum_{j\in\Z^*}\langle k+jN\rangle^{1+2\eps}|\widehat{f}_{k+jN}^{(0)}|^2\Big(\frac{N}{\langle N-|k|\rangle^{1+2\eps}}+\int_N^{+\infty}\frac{\dd x}{{\langle x-|k|\rangle^{1+2\eps}}}\Big)\nonumber\\
   &\le\sum_{1\le |k|\le n}\frac{2\langle n\rangle^{2s}}{N} \sum_{j\in\Z^*}\langle k+jN\rangle^{1+2\eps}|\widehat{f}_{k+jN}^{(0)}|^2\Big(\frac{N}{\langle n\rangle^{1+2\eps}}+\int_n^{+\infty}\frac{\dd x}{{\langle x\rangle^{1+2\eps}}}\Big)\nonumber\\
   &\le C[1+\tfrac{1}{\eps}]\frac{\langle n\rangle^{2s}}{n^{1+2\eps}}\norm{f}_{H^{1/2+\eps}}^2\nonumber.
\end{align}

At last
\[
	\norm{(1-P_n)f}_{H^{s}}\le \frac{1}{\langle n\rangle^{s-1-\eps}}\norm{f}_{H^{1+\eps}}.
\]
Putting everything together yields $\norm{\cI_N(f)-f}_{H^{s}}\le \frac{C[1+\eps^{-1/2}]}{N^{1+\eps-s}}\norm{f}_{H^{1+\eps}}$. \hfill $\qed$

\subsection{Proof of Lemma~\ref{lem:compare_piece_cst_T}}
	A straightforward computation gives:
	\[
		\forall\,k\in\Z,\quad \wh{\Sigma}_{N}(k):=\frac{1}{2\pi}\int_{\Ss}\boldsymbol{\Sigma}_N(z)\overline{z}^k\dd z=\text{sinc}\big( \tfrac{k\pi}{N}\big)\wh{S}_N(k\!\!\!\!\!\mod N),
	\]
	where $\wh{S}_N(k')$ is the $k'$-th Fourier coefficient of $\wt{\Sb}_N$ as given by \eqref{eq:def_f_tilde}, and $k\!\!\!\mod N$ denotes the unique integer in $[-n,n]$ equal to $k$ modulo $N$.
	We get:
	\begin{align}\label{eq:compare_piecewise_cst}
		\norm{\wt{\Sb}_N-\boldsymbol{\Sigma}_N}_{L^2(\Ss)^3}^2&=\sum_{k=-n}^n |\wh{S}_{N}(k)|^2\Big( \big( 1-\text{sinc}\big( \tfrac{k\pi}{N}\big)\big)^2+\sum_{j\in\Z^*} \text{sinc}\big( \tfrac{k\pi}{N}+j\pi\big)^2\Big)\\
											&=	\sum_{k=-n}^n |\wh{S}_{N}(k)|^2\Big( \big( 1-\text{sinc}\big( \tfrac{k\pi}{N}\big)\big)^2+ 1-\text{sinc}\big( \tfrac{k\pi}{N}\big)^2\Big)\nonumber\\
												&=2\sum_{k=-n}^n |\wh{S}_{N}(k)|^2 \big( 1-\text{sinc}\big( \tfrac{k\pi}{N}\big)\big).\nonumber
	\end{align}
	We have used $\sum_{j\in\Z}\text{sinc}\big(x+j\pi\big)^2=1$. We get $\norm{\wt{\Sb}_N-\boldsymbol{\Sigma}_N}_{L^2(\Ss)^3}^2=\underset{N\to\infty}{o}(\norm{\wt{\Sb}_N}_{H^\eps})\to 0$. \hfill $\qed$

\subsection{Uniqueness of Strong Solutions}\label{sec:uniqueness}

We have the following uniqueness result for strong solutions of (HWM), where we do not necessarily assume that the initial data take values on the unit sphere $\Ss^2$.

\begin{lemma}\label{lem:unique_reg}
 Let $I=[0,a)$ with some $a>0$. Given a regular (real-valued) initial datum $\Sb(0)\in  H^{3/2+\eps}(\Ss)^3$ with some $\eps >0$, there exists at most one solution of \emph{(HWM)} in the class $\mathrm{C}^1(I;H^{1/2}(\Ss)^3)\cap \mathrm{C}(I;H^{3/2+\eps}(\Ss)^3)$.
\end{lemma}

\begin{proof}
Let $\Sb_1$ and $\Sb_2$ be two such solutions. We derive a Gr\"{o}nwall estimate for the norm $\norm{\Sb_1-\Sb_2}_{H^{1/2}}^2$ for the difference of two solutions. First, we notice
\begin{align} \label{eq:diff_H_dot_un_demi}
\tfrac{1}{2}\partial_t\norm{\Sb_1-\Sb_2}_{\dot{H}^{1/2}}^2 & =
  \Big\langle|\nabla|(\Sb_1-\Sb_2),\,|\nabla|\Sb_1\times \Sb_1-|\nabla| \Sb_2\times \Sb_2\Big\rangle_{L^2} \\
  & =\Big\langle|\nabla|(\Sb_1-\Sb_2),\,|\nabla|\big(\tfrac{\Sb_1+\Sb_2}{2}\big)\times (\Sb_1-\Sb_2)\Big\rangle_{L^2} \nonumber \\
  & = \Big \langle |\nabla|^{1/2} (\Sb_1 - \Sb_2), \, |\nabla|^{1/2} \left ( |\nabla|\big(\tfrac{\Sb_1+\Sb_2}{2}\big)\times (\Sb_1-\Sb_2)\right ) \Big\rangle_{L^2} . \nonumber 
\end{align}
Using the fractional Leibniz rule (Kato--Ponce--Vega commutator estimate) we find
$$
 \partial_t\norm{\Sb_1-\Sb_2}_{\dot{H}^{1/2}}^2 \lesssim \norm{\Sb_1-\Sb_2}_{H^{1/2}}^2\norm{\Sb_1+\Sb_2}_{H^{3/2+\eps}}.
 $$ 
Similarly we find
\begin{align}\label{eq:diff_L_2}
\tfrac{1}{2}\partial_t\norm{\Sb_1-\Sb_2}_{L^{2}}^2 & =\cip{\Sb_1-\Sb_2}{|\nabla| \Sb_1\times \Sb_1-|\nabla| \Sb_2\times \Sb_2}_{L^2} \\
& =\Big\langle \Sb_1-\Sb_2,\, |\nabla|(\Sb_1-\Sb_2)\times \tfrac{\Sb_1+\Sb_2}{2}\Big\rangle_{L^2} \nonumber \\
 & = \Big\langle |\nabla|^{1/2}(\Sb_1-\Sb_2),\, |\nabla|^{1/2}\big( \tfrac{\Sb_1+\Sb_2}{2}\times (\Sb_1-\Sb_2)\big)\Big\rangle_{L^2} \nonumber \\
 & \lesssim  \norm{\Sb_1-\Sb_2}_{H^{1/2}}^2\norm{\Sb_1+\Sb_2}_{H^{1/2}\cap L^\infty}. \nonumber 
\end{align}
From these estimates it is straightforward to show that $\Sb_1 \equiv \Sb_2$ on $[0,a)$ if $\Sb_1(0)=\Sb_2(0)$ at time $t=0$ by a Gr\"onwall type argument.
\end{proof}

\end{appendix}


\begin{thebibliography}{Ha-88}


\bibitem{BaBeTa-03} Babelon, O. and Bernard, D. and Talon, M., {\em Introduction to classical integrable systems}, Cambridge University Press, Cambridge, 2003.

\bibitem{BeKlLa-20} Berntson, B. K. and Klabbers, R. and Langmann, E., {\em Multi-solitons of the half-wave maps equation and Calogero-Moser spin-pole dynamics}, Preprint available at arXiv:2006.16826 (2020).
	
 
\bibitem{GeLe-18} G\'erard, P. and Lenzmann, E., {\em A {L}ax pair structure for the half-wave maps equation}, Lett. Math. Phys.~108 (2018), 1635--1648.


\bibitem{GiHe-84} Gibbons, J. and Hermsen, T., {\em A generalisation of the {C}alogero-{M}oser system}, Phys. D 11 (1984), 337--348.


\bibitem{Ha-88} Haldane, F. D. M., {\em Exact Jastrow-Gutzwiller resonating-valence-bond ground state of the spin-$\frac{1}{2}$ antiferromagnetic Heisenberg chain with 1/${\mathrm{r}}^{2}$ exchange}, Phys. Rev. Lett. 60 (1988), 635--638.


\bibitem{KiLeSt-13} Kirkpatrick, K. and Lenzmann, E. and Staffilani, G., {\em  On the continuum limit for discrete {NLS} with long-range
              lattice interactions}, Comm. Math. Phys. 317 (2013), 563--591.


		
\bibitem{KrBaBiTa-95} Krichever, I. M. and Babelon, O. and Billey, E. and Talon, M., {\em Spin generalization of the {C}alogero-{M}oser system and the
              matrix {KP} equation}, Topics in topology and mathematical physics, Amer. Math. Soc. Transl. Ser. 170 (1995), 83--119.
              
\bibitem{KrSi-18} Krieger, J. and Sire, Y., {\em Small data global regularity for half-wave maps}, Anal. PDE 11 (2018), 661--682.


\bibitem{Le-18} Lenzmann, E., {\em On short primer on the Half-Wave Maps Equation}, Journ\'ees EDP (2018), Exp. No. 4, 12pp.
\bibitem{LeSc-18} Lenzmann, E., and Schikorra, A., {\em On energy--critical half--wave maps into $\Ss^2$}, Invent. Math. 273 (2018), 1-82.

\bibitem{LeSo-20} Lenzmann, E. and Sok, J., {\em On the continuum limit for classical spin systems with long-range interactions}, Work in preparation (2020).

\bibitem{PuGu-13} Pu, X. and Guo, B., {\em Well-posedness for the fractional {L}andau-{L}ifshitz equation without {G}ilbert damping}, Calc. Var. Partial Differential Equations 46 (2013), 441-460.


\bibitem{Sh-88} Shastry, B. S., {\em Exact solution of an S=1/2 Heisenberg antiferromagnetic chain with long-ranged interactions}, Phys. Rev. Lett. 60 (1988), 639--642.

\bibitem{SuSuBa-86} Sulem, P.-L. and Sulem, C. and Bardos, C., {\em On the continuous limit for a system of classical spins}, Comm. Math. Phys. 107 (1986), 431--454.

\bibitem{Wo-85} Wojciechowski, S., {\em An integrable marriage of the Euler equations with the Calogero-Moser system}, Physics Letters A 111 (1985), 101--103.

\bibitem{ZhSt-15} Zhou, T. and Stone, M., {\em Solitons in a continuous classical Haldane--Shastry spin chain,} Phys. Lett. A 379 (2015), 2817--2825.

\end{thebibliography}
\end{document}